\documentclass[reqno,10pt]{amsart}    

\title[Mean curvature flow in higher codimension]{Mean curvature flow in higher codimension\\ - Introduction and survey -}
\author[Knut Smoczyk]{\sc Knut Smoczyk}
\address{
Leibniz Universit\"at Hannover,
Institut f\"ur Differentialgeometrie,    
Welfengarten 1,
30167 Hannover,
Germany}
\email{smoczyk@math.uni-hannover.de} 
\thanks{This survey is a contribution within the framework of the priority program ``Globale Differentialgeometrie", DFG-SPP 1154, supported by the German Science Foundation (DFG)}
  
\subjclass[2000]{Primary 53C44; 
}   
\keywords{introduction, survey, mean curvature flow, codimension, Lagrangian}
\date{\today} 

\usepackage{amscd,amsfonts,mathrsfs,amsthm,enumerate}
\usepackage{amssymb, amsmath}          
\usepackage[alphabetic, sorted-cites]{amsrefs} 
\def\real     #1{{\mathbb R^{#1}}}
\def\complex  #1{{\mathbb C^{#1}}}

\def\dd       #1#2#3{{#1}_{#2#3}}
\def\ddd      #1#2#3#4{{#1}_{#2#3#4}}
\def\dddd     #1#2#3#4#5{{#1}_{#2#3#4#5}}

\def\uu       #1#2#3{{#1}^{#2#3}}
\def\uuu      #1#2#3#4{{#1}^{#2#3#4}}
\def\uuuu     #1#2#3#4#5{{#1}^{#2#3#4#5}}
\def\ud       #1#2#3{{#1}^{#2}_{\phantom{#2}{#3}}}
\def\du       #1#2#3{{#1}_{#2}^{\phantom{#2}{#3}}}
\def\udu      #1#2#3#4{{#1}^{{#2}\phantom{#3}{#4}}_{\phantom{#2}{#3}}}

\def\ddu      #1#2#3#4{{#1}_{#2#3}^{\phantom{#2#3}{#4}}}
\def\udd      #1#2#3#4{{#1}^{#2}_{\phantom{#2}{#3#4}}}
\def\duu      #1#2#3#4{{#1}_{#2}^{\phantom{#2}{#3#4}}}

\def\uddd     #1#2#3#4#5{#1^#2_{\phantom{#2}#3#4#5}}

\def\dddu     #1#2#3#4#5{#1_{#2#3#4}^{\phantom{#2#3#4}#5}}

\def\dudu     #1#2#3#4#5{#1_{#2\phantom{#3}#4}^
              {\phantom{#2}#3\phantom{#4}#5}}

\def\dt       {\frac{d}{dt}\,}

\def\equationcolor {\color{black}}
\def\textcolor     {\color{black}}

\def\bcoleq    {\begin{equation}\equationcolor}
\def\ecoleq    {\textcolor\end{equation}}

\def\bcoleqn   {\equationcolor\begin{eqnarray}}
\def\ecoleqn   {\end{eqnarray}\textcolor}

\newtheorem{theorem}{Theorem}[section]   
\newtheorem{lemma}[theorem]{Lemma}   
\newtheorem{corollary}[theorem]{Corollary}   
\newtheorem{proposition}[theorem]{Proposition}   
\newtheorem{remark}[theorem]{Remark}   
  
\theoremstyle{definition}   
\newtheorem{definition}[theorem]{Definition}   
\newtheorem{example}[theorem]{Example}

\newcommand{\bfig}{\begin{figure}}
\newcommand{\efig}{\end{figure}}

\makeatletter
\def\pproof#1{\@ifnextchar[\opargproof
{\opargproof[\it Proof of #1.]}}
\def\opargproof[#1]{\par\noindent {\bf #1 }}

\makeatother

\begin{document}
\begin{abstract}
In this text we outline the major techniques, concepts and results in mean curvature flow with a focus
on higher codimension. In addition we include a few novel results and some material that cannot be found elsewhere.
\end{abstract}
\maketitle

\section{Mean curvature flow}
Mean curvature flow is perhaps the most important geometric evolution equation of submanifolds in Riemannian manifolds. Intuitively, a family of smooth submanifolds evolves under mean curvature flow, if the velocity at each point of the submanifold is given by the mean curvature vector at that point. For example, round spheres in euclidean space
evolve under mean curvature flow while concentrically shrinking inward until they collapse in finite time to a single point, the common center of the spheres.

Mullins \cite{Mullins56} proposed mean curvature flow to model the formation of grain boundaries in annealing metals. Later the evolution of submanifolds by their mean curvature has been studied by
Brakke \cite{Brakke} from the viewpoint of geometric measure theory. Among the first authors
who studied the corresponding nonparametric problem were Temam \cite{Temam} in the late 1970's
and Gerhardt \cite{Gerhardt} and Ecker \cite{Ecker82} in the early 1980's.
Pioneering work was done by Gage \cite{Gage84}, Gage \& Hamilton \cite{Gage-Hamilton86} and Grayson \cite{Grayson87} 
who proved that the curve shortening flow (more precisely,  the ``mean"
curvature flow of curves in $\real{2}$) shrinks embedded closed curves to ``round" points. In his seminal paper Huisken   \cite{Huisken84} proved that closed convex hypersurfaces in euclidean space $\real{m+1}, m>1$ contract to single round points in finite time (later he extended his
result to hypersurfaces in Riemannian manifolds that satisfy a suitable stronger convexity, see \cite{Huisken86}). Then, until the mid 1990's, most authors who studied mean curvature flow mainly considered hypersurfaces, both in euclidean and Riemannian manifolds, whereas mean curvature flow in higher codimension did not play a great
role. There are various reasons for this, one of them is certainly the much different geometric situation of submanifolds in higher codimension
since the normal bundle and the second fundamental tensor are more
complicated. But also the analysis becomes more involved
and the algebra of the second fundamental tensor is much more subtle since for hypersurfaces there usually exist more
scalar quantities related to the second fundamental form
than in case of submanifolds in higher codimension. Some
of the results previously obtained for mean curvature flow of hypersurfaces carry
over without change to submanifolds of higher codimension but many do not and
in addition even new phenomena occur.

Among the first results in this direction are the results on mean
curvature flow of space curves by Altschuler and Grayson \cites{Altschuler91, Altschuler-Grayson92}, 
measure-theoretic approaches to higher codimension mean curvature flows
by Ambrosio \& Soner \cite{Ambrosio-Soner97},  existence and convergence results 
for the Lagrangian mean curvature flow \cites{Smoczyk96, Smoczyk99, Smoczyk02, Thomas-Yau02}, mean curvature flow of symplectic surfaces in codimension two \cites{Chen-Li02, Wang02}
and long-time existence and convergence results of graphic mean curvature flows in
higher codimension \cites{Chen-Li-Tian02, Smoczyk-Wang02, Smoczyk04, Wang02, Xin08}. Recently there has been done quite some work on the formation
and classification of singularities in mean curvature flow
\cites{Anciaux, Castro-Lerma09, Chau-Chen-He09, Clutterbuck-Schnuerer-Schulze07, Colding-Minicozzi09, GSSZ07, Han-Li09, Huisken-Sinestrari09, Joyce-Lee-Tsui08,  Le-Sesum10, Le-Sesum10b, Liu-Xu-Ye-Zhao, Schoen-Wolfson03}, partially motivated
by Hamilton's and Perelman's \cites{Hamilton93, Perelman02, Perelman03a, Perelman03b} work on the Ricci flow that in many ways behaves akin to
the mean curvature flow and vice versa.

The results in mean curvature flow can be roughly grouped into two categories: The first category contains results that hold (more or less) in general,
i.e. that are independent of dimension, codimension or the ambient space.
In the second class we find results that are adapted to more specific geometric situations, like results for hypersurfaces, Lagrangian or symplectic submanifolds, graphs, etc..

Our aim in this article is twofold. We first want to summarize the most important properties of
mean curvature flow that hold in any dimension,
codimension and ambient space (first category). In the second part of this exposition we will give a - certainly
incomplete and not exhaustive -, overview on more specific results in higher codimension, like an overview on the Lagrangian mean curvature flow or the mean curvature flow of graphs (part of the second category). Graphs and Lagrangian submanifolds certainly
form the best understood subclasses of mean curvature flow in higher codimension.

 In addition this article is intended as an introduction to 
mean curvature flow for the beginner and we will derive the most relevant geometric structure and evolution equations
in a very general but consistent form that is rather hard to find in the literature.
However, there are several nice monographs on mean curvature flow, a well written introduction to the regularity of mean curvature flow of hypersurfaces is 
\cite{Ecker04}. For the curve shortening flow see \cite{Chou-Zhu01}. For mean curvature flow in higher codimension there exist some lecture notes by Wang \cite{Wang08b}.

Let us now turn our attention to the mathematical definition of mean curvature flow.
Suppose $M$ is a differentiable manifold of dimension $m$, $T>0$ a real number and $F:M\times[0,T)\to (N,g)$ a smooth time dependent family of immersions of $M$ into a Riemannian manifold $(N,g)$ of dimension $n$, i.e. $F$ is smooth and each 
\[F_t:M\to N\,,\quad F_t(p):=F(p,t)\,,\quad t\in[0,T)\]
is an immersion. If $F$ satisfies the evolution equation
\begin{gather}\label{mcf}
\frac{d F}{dt}(p,t)=\overrightarrow H(p,t)\,,\quad\forall p\in M, t\in[0,T)\,,\tag{MCF}
\end{gather}
where $\overrightarrow H(p,t)\in T_{F(p,t)}N$ is the mean curvature vector of the immersion $F_t$ at
$p$ (or likewise of the submanifold $U_t:=F_t(U)$ at $F_t(p)$, if for some $U\subset M$, $F_{t|U}$ is an embedding), then we say that
$M$ evolves by mean curvature flow in $N$ with initial data $F_0:M\to N$. 
As explained in section \ref{sec mean}, the mean curvature vector field can be defined for any immersion into a Riemannian manifold (or more generally for any space-like immersion into a pseudo-Riemannian manifold; in this survey we will restrict to
the Riemannian mean curvature flow) and it is the negative $L^2$-gradient of the volume functional $\operatorname{vol}:\mathscr{I}\to\real{}$ on the space $\mathscr{I}$ of immersions of $M$ into $(N,g)$.
Hence mean curvature flow is the steepest descent or negative $L^2$-gradient flow of the volume functional and formally equation (\ref{mcf}) makes sense for any immersed submanifold in a Riemannian manifold.  Therefore, following Hadamard, given an initial immersion $F_0:M\to N$ one is interested in the well-posedness of equation (\ref{mcf}) in the sense of
\begin{enumerate}[I.)]
\item
Does a solution exist?
\item
Is it unique?
\item
Does it behave continuously in some suitable topology?
\end{enumerate}
In addition, once short-time existence is established on some maximal time interval $[0,T), T\in(0,\infty]$, one wants to study the behavior of the flow
and in particular of the evolving immersed submanifolds $M_t:=F_t(M)$ as $t\to T$. Either singularities
of some kind will form and one might then study the formation of singularities in more details - with
possible significant geometric implications - or the flow has a long-time solution. In such a case
convergence to some nice limit (e.g. stationary, i.e. a limit with vanishing mean curvature) would be rather expected but in general will not hold a-priori.

In the most simplest case, i.e. if the dimension of $M$ is one, mean curvature flow is called curve shortening flow. In many contributions to the theory of mean curvature flow one assumes that $M$ is a smooth {\sl closed} manifold. The reason is, that one key technique in mean curvature flow (or more generally in the theory of parabolic geometric evolution equations) is the application of the maximum principle and in absence of compactness the  principle of ``first time violation" of a stated
inequality simply does not hold. But even for complete non-compact submanifolds there are
powerful techniques, similar to the maximum principle, that can be applied in
some situations. In the complete case one of the most important tools is the monotonicity formula found by
Huisken \cite{Huisken90}, Ecker \& Huisken \cite{Ecker-Huisken89} and Hamilton \cite{Hamilton93a}
and that equally well applies to mean curvature flow in higher codimension. Ecker \cite{Ecker01}
proved a beautiful local version of the
monotonicity formula for hypersurfaces
and another local monotonicity for evolving Riemannian
manifolds has been found recently by Ecker, Knopf, Ni and Topping \cite{Ecker-Knopf-Ni-Topping08}.

There are some very important contributions  to the regularity theory of mean curvature flow by White \cites{White05, White09}
that apply in all codimensions. For example in \cite{White05} he proves uniform curvature bounds 
of the euclidean mean curvature flow in regions of space-time where the Gaussian density ratios are close to $1$. With this result one can often 
exclude finite time singularities and prove long-time existence of the flow (see for example \cites{Wang02, Medos-Wang09}).

For simplicity and since some techniques and results do not hold for complete non-compact manifolds we will always assume in this article, unless otherwise agreed, that $M$ is an oriented {\sl closed} smooth manifold.

The organization of the survey is as follows: In section \ref{sec structure} we will review the geometric structure equations for immersions in Riemannian manifolds and we will introduce
most of our terminology and notations that will be used throughout the paper.
In particular we will mention the explicit formulas in the case of Lagrangian submanifolds in K\"ahler-Einstein manifolds.
For most computations we will use the Ricci calculus and apply the Einstein convention to sum over repeated indices. 
In section \ref{sec general} we will 
summarize those results that hold in general (first category).
The section is subdivided into four subsections. In the first
subsection \ref{sec existence} we will show that the mean curvature flow is 
a quasilinear (degenerate) parabolic system and we will treat the existence and
uniqueness problem.
In subsection \ref{sec evolution} we derive the evolution equations of the most important 
geometric quantities in the general situation, i.e. for immersions
in arbitrary Riemannian manifolds. In this general form these
formulas are hard to find in the literature and one can later
easily derive all related evolution equations from them that
occur in special situations like evolution equations for tensors 
that usually
appear in mean curvature flow of hypersurfaces, Lagrangian submanifolds or graphs. 
In subsection \ref{sec longtime} we recall general results concerning long-time existence of solutions. In the final subsection
\ref{sec singularities} of this section we explain the two types
of singularities that appear in mean curvature flow and discuss
some rescaling techniques. Moreover we will recall some of the results that have been obtained in the classification of solitons.
Section \ref{sec special} is on more specific results in higher codimension, the first subsection treats the Lagrangian mean
curvature flow and in the last and final subsection of this
article we give an overview of the results in mean
curvature flow of graphs.

\section{The geometry of immersions}\label{sec structure}
\subsection{Second fundamental form and mean curvature vector}\label{sec mean}~

In this subsection we recall the definition of the second fundamental form and mean curvature vector
of an immersion and we will introduce most of our notation.

Let $F:M\to (N,g)$ be an immersion of an $m$-dimensional differentiable manifold $M$ into
a Riemannian manifold $(N,g)$ of dimension $n$, i.e. $F$ is smooth and the pull-back $F^*g$ defines a Riemannian metric on $M$.  The number $k:=n-m\ge 0$ is called the codimension of the
immersion. 

For $p\in M$ let
$$T_p^\perp M:=\{\nu\in T_{F(p)}N:g(\nu,DF_{|p}(W))=0,\,\forall W\in T_pM\}$$
denote the normal space of $M$ at $p$ and $T^\perp M$ the associated normal bundle.
By definition, the normal bundle of $M$ is a sub-bundle of rank $k$ of the pull-back bundle $F^*TN=\bigcup_{p\in M}T_{F(p)}N$ over $M$. Using the differential of $F$ we thus have a splitting
$$T_{F(p)}N=DF_{|p}(T_pM)\oplus T_p^\perp M\,.$$
The differential $DF$ can be considered as a $1$-form on $M$ with values in $F^*TN$, i.e.
\begin{align}
DF\in\Gamma(F^*TN\otimes T^*M)=:\Omega^1(M,F^*TN)\,,\\\nonumber
\quad T_pM\ni V\mapsto DF_{|p}(V)\in T_{F(p)}N\,.\nonumber
\end{align}

The Riemannian metric $F^*g$ is also called the first fundamental form on $M$. 
In an obvious way the metrics $g$ and $F^*g$ induce Riemannian metrics on all bundles formed
from products of $TM,T^*M, T^\perp M, F^*TN, TN,$ and $T^*N$ and in the sequel we will
often denote all such metrics simply by the usual brackets $\langle\cdot,\cdot\rangle$ for an
inner product.

Similarly the Levi-Civita connection $\nabla$ on $(N,g)$ induces connections on the bundles
$TM, T^*M, T^\perp M, F^*TN$ and products hereof. Since the precise definition of these connections will be crucial in the understanding of the second fundamental form, the mean curvature vector and later also of the evolution equations, we will briefly recall them. The
connection $\nabla^{TM}$ on $TM$ can be obtained in two equivalent ways: either as the Levi-Civita connection of the induced metric $F^*g$ on $TM$ or else by projection of the ambient connection
to the tangent bundle, more precisely via the formula
$$DF(\nabla^{TM}_XY):=\nabla^\top_{DF(X)}\overline{DF(Y)}\,,\quad X,Y\in TM\,,$$
where ${}^\top$ denotes the projection onto $DF(TM)$
and $\overline{DF(Y)}$ is an arbitrary (local) smooth extension of $DF(Y)$.
The connection $\nabla^{T^*M}$
on $T^*M$ is then simply given by the dual connection of $\nabla^{TM}$.
Similarly one obtains the connection $\nabla^{F^*TN}$ on $F^*TN$ via the formula
$$\nabla^{F^*TN}_XV:=\nabla_{DF(X)}\overline{V}\,,$$
for any smooth section $V\in\Gamma(F^*TN)$ and finally the connection $\nabla^\perp$ on
the normal bundle is given by projection
$$\nabla^\perp_X\nu:=\left(\nabla^{F^*TN}_X\nu\right)^\perp$$
for $\nu\in\Gamma(T^\perp M)\subset\Gamma(F^*TN)$. Since the connections $\nabla^{TM}$, $\nabla^{T^*M}$, $\nabla^{F^*TN}$ and their associated product connections on product bundles
over $M$ formed from the factors $TM, T^*M, F^*TN$ are induced
by $\nabla$, it is common (and sometimes confusing) to denote all of them by the same symbol $\nabla$.
Since $T^\perp M$ is a sub-bundle of $F^*TN$,
one can consider a section $\nu\in\Gamma(T^\perp M)$ also as an element of $\Gamma(F^*TN)$
and hence one can apply both connections $\nabla^\perp$ and $\nabla=\nabla^{F^*TN}$ to
them, i.e. we 
will write $\nabla_X\nu$ ($=\nabla_X^{F^*TN}\nu$), if we consider $\nu$ as a section in 
$F^*TN$ and $\nabla^\perp_X\nu$, if $\nu$ is considered as a section in the normal bundle
$T^\perp M$.  The same holds, if we consider sections in product
bundles that contain $T^\perp M$ as a factor.

If we apply the resulting connection $\nabla$ on $F^*TN\otimes T^*M$ to $DF$, we obtain - by
definition - the second fundamental tensor
$$A:=\nabla DF\in\Gamma(F^*TN\otimes T^*M\otimes T^*M)\,.$$
It is then well-known that the second fundamental tensor is symmetric
\begin{equation}\label{sec 1}
A(X,Y)=(\nabla_XDF)(Y)=(\nabla_YDF)(X)=A(Y,X)
\end{equation}
and normal in the sense that
\begin{equation}\label{sec 2}
\langle A(X,Y),DF(Z)\rangle=0\,,\quad\forall X,Y,Z\in TM\,.
\end{equation}
Therefore in particular $A\in\Gamma(T^\perp M\otimes T^*M\otimes T^*M)$.

Taking the trace of $A$ gives the mean curvature vector field
\begin{equation}\label{sec 3}
\overrightarrow H:=\operatorname{trace}A=\sum_{i=1}^mA(e_i,e_i)\,,
\end{equation}
where $(e_i)_{i=1,\dots,m}$ is an arbitrary orthonormal frame of $TM$. Hence, since
$A$ is normal, we obtain a canonical section $\overrightarrow H\in\Gamma(T^\perp M)$ in the normal bundle of the immersion $F:M\to N$.

\subsection{Structure equations}~

The second fundamental tensor is a curvature quantity that determines how curved the immersed
submanifold $F(M)$ given by an immersion $F:M\to N$ lies within the ambient manifold $(N,g)$.
According to this we have a number of geometric equations that relate the second fundamental
tensor to the intrinsic curvatures of $(M,F^*g)$ and $(N,g)$. 

Let $\nabla$ be a connection on a vector bundle $E$ over a smooth manifold $M$. Our convention for the curvature tensor $R^{E,\nabla}\in\Omega^2(M,E)$ w.r.t. $\nabla$ is
$$R^{E,\nabla}(X,Y)\sigma:=(\nabla_X\nabla_Y-\nabla_Y\nabla_X-\nabla_{[X,Y]})\sigma\,,\quad\forall X,\,Y\in TM,\sigma\in\Gamma(E)\,.$$
Moreover, if $E$ is a bundle with bundle metric $\langle\cdot,\cdot\rangle$, then we set
$$R^{E,\nabla}(\mu,\sigma,X,Y):=\langle\mu,R^{E,\nabla}(X,Y)\sigma\rangle\,,\quad\forall X,Y\in TM,\,\sigma,\mu\in E\,.$$

We denote the curvature tensors $R^{TM,\nabla}$ and $R^{TN,\nabla}$ by
$R^M$ resp. $R^N$. Letting
$$(\nabla_XA)(Y,V):=\nabla_X(A(Y,V))-A(\nabla_XY,V)-A(Y,\nabla_XV)$$
the Codazzi equation is
\begin{eqnarray}
&&(\nabla_XA)(Y,V)-(\nabla_YA)(X,V)\nonumber\\
&&=R^N(DF(X),DF(Y))DF(V)-DF(R^M(X,Y)V)\,.\label{struc 2}
\end{eqnarray}
Note that $\nabla$ denotes the full connection, i.e. here we consider $A$ as a section
in $F^*TN\otimes T^*M\otimes T^*M$ and not in $T^\perp M\otimes T^*M\otimes  T^*M$.
Later we will sometimes consider $A$ as a section in $T^\perp M\otimes T^*M\otimes  T^*M$
and then we will also use the connection on the normal bundle instead, so that in this case we
write $(\nabla^\perp_XA)(Y,V)=\left((\nabla_XA)(Y,V)\right)^\perp$. In terms of $\nabla^\perp$
the Codazzi equation becomes
\begin{equation}\label{struc 2b}
(\nabla^\perp_XA)(Y,V)-(\nabla^\perp_YA)(X,V)=\left(R^N(DF(X),DF(Y))DF(V)\right)^\perp.
\end{equation}

From
$$\langle A(Y,V),DF(W)\rangle =0\,,\quad\forall Y,V,W\in TM$$
we get
\begin{equation}\label{struc 3}
\langle(\nabla_XA)(Y,V),DF(W)\rangle=-\langle A(Y,V),A(X,W)\rangle\,.
\end{equation}

From these equations we obtain Gau\ss\ equation (Theorema Egregium):
\begin{align}
R^M(X,Y,V,W)&=R^N(DF(X),DF(Y),DF(V),DF(W))\label{struc 1}\\
&+\langle A(X,V),A(Y,W)\rangle-\langle A(X,W),A(Y,V)\rangle\,.\nonumber
\end{align}

Finally, we have Ricci's equation. If $\nu\in T^\perp M$ and $X,Y\in TM$ then the following holds:
\begin{align}
&R^\perp(X,Y)\nu=(R^N(DF(X),DF(Y))\nu)^\perp\nonumber\\
&-\sum_{i=1}^m\bigl(\langle\nu,A(X,e_i)\rangle A(Y,e_i)
-\langle\nu,A(Y,e_i)\rangle A(X,e_i)\bigr)\,,\label{struc 4}
\end{align}
where $(e_i)_{i=1,\dots,m}$ is an arbitrary orthonormal frame of $TM$ and $R^\perp=R^{T^\perp M,\nabla^\perp}$ denotes the curvature tensor of the normal bundle of $M$. Note that the
Codazzi equation is useless in dimension one (i.e. for curves) and that Ricci's equation is useless
for hypersurfaces, i.e. in codimension one.

\subsection{Tensors in local coordinates}\label{sec loc}~

For computations one often needs local expressions of tensors. Whenever
we use local expressions and $F:M\to N$ is an immersion
we make the following general assumptions and notations
\begin{enumerate}[i)]
\item
$(U,x,\Omega)$ and $(V,y,\Lambda)$ are local coordinate charts around $p\in U\subset M$ and
$F(p)\in V\subset N$ such that $F_{|U }:U\to F(U)$ is an embedding and such that
$F(U)\subset V$. 
\item
From the coordinate functions
$$(x^i)_{i=1,\dots,m}:U\to\Omega\subset\real{m}\,,\quad (y^\alpha)_{\alpha=1,\dots,n}:V\to\Lambda\subset\real{n}$$
we obtain a local expression for $F$,
$$y\circ F\circ x^{-1}:\Omega\to\Lambda\,,\quad F^\alpha:=y^\alpha\circ F\circ x^{-1},\,\quad\alpha=1,\dots,n.$$
\item
The Christoffel symbols of the Levi-Civita connections on $M$ resp. $N$ will be denoted
$$\Gamma^i_{jk},\quad i,j,k=1,\dots,m\,,\quad\text{resp.}\quad\Gamma^\alpha_{\beta\gamma},\quad\alpha,\beta,\gamma=1,\dots, n\,.$$
\item
All indices referring to $M$ will be denoted by Latin minuscules and those related to $N$ by Greek minuscules.
Moreover, we will always use the Einstein convention to sum over repeated indices from $1$ to
the respective dimension.
\end{enumerate}

Then the local expressions for $g, DF,\,F^*g$ and $A$ are
$$g=\dd g\alpha\beta dy^\alpha\otimes dy^\beta\,,$$
$$DF=\ud F\alpha i\frac{\partial}{\partial y^\alpha}\otimes dx^i\,,\quad \ud F\alpha i:=\frac{\partial F^\alpha}{\partial x^i}\,,$$
$$F^*g=\dd gij dx^i\otimes dx^j\,,\quad\dd gij:=\dd g\alpha\beta \ud F\alpha i \ud F\beta j,$$
and
$$A=\dd Aijdx^i\otimes dx^j=\udd A\alpha ij\frac{\partial}{\partial y^\alpha}\otimes dx^i\otimes dx^j\,,$$
where the coefficients $\udd A\alpha ij$ are given by Gau\ss' formula
\begin{equation}\label{sec 4}
\udd A\alpha ij=\frac{\partial^2 F^\alpha}{\partial x^i\partial x^j}-\Gamma^k_{ij}\frac{\partial F^\alpha}{\partial x^k}+\Gamma^\alpha_{\beta\gamma}\frac{\partial F^\beta}{\partial x^i}\frac{\partial F^\gamma}{\partial x^j}\,.
\end{equation}
Let $(\uu gij)$ denote the inverse matrix of  $(\dd gij)$ so that $\uu gik\dd gkj=\ud \delta ij$
gives the Kronecker symbol. $(\uu gij)$ defines the metric on $T^*M$ dual to $F^*g$.
For the mean curvature vector we get
\begin{equation}\label{sec 5}
\overrightarrow H=H^\alpha\frac{\partial}{\partial y^\alpha}\,,\quad H^\alpha:=\uu gij \udd A\alpha ij\,.
\end{equation}
Gau\ss' equation (\ref{struc 1}) now becomes
\begin{equation}
R_{ijkl}=R_{\alpha\beta\gamma\delta}\ud F\alpha i\ud F\beta j\ud F\gamma k\ud F\delta l+\dd g\alpha\beta(\udd A\alpha ik\udd A\beta jl-\udd A\alpha il\udd A\beta jk)\,,\label{struc 1l}
\end{equation}
where the notation should be obvious, e.g.
$$R_{ijkl}=R^M\left(\frac{\partial}{\partial x^i},\frac{\partial}{\partial x^j},\frac{\partial}{\partial x^k},\frac{\partial}{\partial x^l}\right)$$
and
$$R_{\alpha\beta\gamma\delta}=R^N\left(\frac{\partial}{\partial y^\alpha},\frac{\partial}{\partial y^\beta},\frac{\partial}{\partial y^\gamma},\frac{\partial}{\partial y^\delta}\right)\,.$$
Note that the choice of the indices already indicates which curvature tensor is used.
In addition we write
$$\nabla A=\nabla_i\udd A\alpha jk\frac{\partial}{\partial y^\alpha}\otimes dx^i\otimes dx^j\otimes dx^k\,,$$
so that
$$(\nabla_{\frac{\partial}{\partial x^i}}A)\left(\frac{\partial}{\partial x^j},
\frac{\partial}{\partial x^k}\right)=\nabla_i\udd A\alpha jk\frac{\partial}{\partial y^\alpha}\,.$$
Similar notations will be used for other covariant derivatives, e.g. $\nabla_i\nabla_j\ud Tkl$ will
denote the coefficients of the tensor $\nabla^2T$ with $T\in\Gamma(TM\otimes T^*M)=\operatorname{End}(TM)$.
The Codazzi equation in local coordinates is
\begin{eqnarray}
\nabla_i\udd A\alpha jk-\nabla_j\udd A\alpha ik
=\uddd R{\alpha}\beta\gamma\delta \ud F\beta k\ud F\gamma i\ud F\delta j
-\uddd Rlkij \ud F\alpha l\,,\label{struc 2l}
\end{eqnarray}
where here and in the following all indices will be raised and lowered using the metric tensors, e.g.
$$\uddd R\alpha\beta\gamma\delta=\uu g\alpha\epsilon\dddd R\epsilon\beta\gamma\delta\,,\quad
\dudu Rkilj=\uu gip\uu gjq\dddd Rkplq\,.$$
Finally, if $(\nu_A)_{A=1,\dots,k:=n-m}$, $\nu_A=\nu_A^\alpha\frac{\partial}{\partial y^\alpha}$, is a local trivialization of $T^\perp M$, then
$$R^\perp\left(\frac{\partial}{\partial x^i},\frac{\partial}{\partial x^j}\right)\nu_A=:(R^\perp)\uddd {}BAij\nu_B$$
and Ricci's equation becomes
\begin{eqnarray}
(R^\perp)\uddd {}BAij\nu_B^\alpha&=&\uddd R\alpha\beta\gamma\delta\nu_A^\beta \ud F\gamma i\ud F\delta j-\uu gkl \uddd R\epsilon\beta\gamma\delta\dd g\epsilon\sigma\nu_A^\beta \ud F\gamma i\ud F\delta j \ud F\sigma k\nonumber\\
&&-\dd g\beta\gamma\uu gkl(\nu_A^\beta A^\gamma_{ik}A^\alpha_{jl}-\nu_A^\beta A^\gamma_{jk}A^\alpha_{il})\,.\label{struc 4l}
\end{eqnarray}
Using the rule for interchanging covariant derivatives and the structure equations one obtains
Simons' identity
\begin{eqnarray}
\nabla_k\nabla_lH^\alpha
&=&\Delta \udd A\alpha kl
+\left(\nabla_\epsilon\uddd R\alpha\beta\gamma\delta+\nabla_\gamma\uddd R\alpha\delta\beta\epsilon\right)
\ud F\epsilon i\ud F\beta l\ud F\gamma k\uu F\delta i\nonumber\\
&&+\uddd R\alpha\beta\gamma\delta
\left(2\udd A\beta ik\ud F\gamma l\uu F\delta i+2\udd A\beta il\ud F\gamma k\uu F\delta i\right.\nonumber\\
&&\left.+H^\delta\ud F\beta l\ud F\gamma k
+\udd A\gamma lk\ud F\beta i\uu F\delta i\right)\nonumber\\
&&-\left(\nabla_k\ud Rpl+\nabla_l\ud Rpk-\nabla^p\dd Rkl\right)\ud F\alpha p\nonumber\\
&&+2\dudu Rkilj\udd A\alpha ij-\ud Rpk\udd A\alpha pl
-\ud Rpl\udd A\alpha pk\,,\label{simons}
\end{eqnarray}
where $\dd Rij=\uu gkl\dddd Rikjl$ denotes the Ricci curvature on $M$. 
If one multiplies Simons' identity (\ref{simons}) with $2\duu A\alpha kl=2\dd g\alpha\epsilon\uu gkm\uu gln\udd A\epsilon mn$, one gets
\begin{eqnarray}
2\langle A,\nabla^2 \overrightarrow H\rangle
&=&\Delta|A|^2-2|\nabla A|^2\nonumber\\
&&
+2\left(\nabla_\epsilon\dddd R\alpha\beta\gamma\delta+\nabla_\gamma\dddd R\alpha\delta\beta\epsilon\right)
\ud F\epsilon i\uu F\beta l\uu F\gamma k\uu F\delta i\udd A\alpha kl\nonumber\\
&&+2\dddd R\alpha\beta\gamma\delta \uuu A\alpha kl
\left(4\udd A\beta ik\ud F\gamma l\uu F\delta i+H^\delta\ud F\beta l\ud F\gamma k
+\udd A\gamma lk\ud F\beta i\uu F\delta i\right)\nonumber\\
&&+4\uuuu Rkilj\langle A_{ij},A_{kl}\rangle-4\uu Rij\langle A_{ik},\du Ajk\rangle\nonumber
\end{eqnarray}
and then since
\begin{eqnarray}
\nabla_iA_{kl}
&=&\nabla^\perp_iA_{kl}+\uu gpq\langle\nabla_iA_{kl},F_p\rangle F_q\nonumber\\
&=&\nabla^\perp_iA_{kl}-\uu gpq\langle A_{kl},\nabla_iF_p\rangle F_q\nonumber\\
&=&\nabla^\perp_iA_{kl}-\uu gpq\langle A_{kl},A_{ip}\rangle F_q\nonumber
\end{eqnarray}
implies
$$|\nabla A|^2=|\nabla^\perp A|^2+\langle A^{ij},A^{kl}\rangle\langle A_{ij},A_{kl}\rangle$$
we obtain with Gau\ss' equation the second Simons' identity
\begin{eqnarray}
2\langle A,\nabla^2 \overrightarrow H\rangle
&=&\Delta|A|^2-2|\nabla^\perp A|^2\nonumber\\
&&+2\langle A^{ij},A^{kl}\rangle\langle A_{ij},A_{kl}\rangle
-4\langle A^{kj},A^{il}\rangle\langle A_{ij},A_{kl}\rangle\nonumber\\
&&-4\langle \overrightarrow H,A^{ij}\rangle\langle A_{ik}, \du Ajk\rangle
+4\langle A^{il},\du Alj\rangle\langle A_{ik}, \du Ajk\rangle\nonumber\\
&&+4\dddd R\alpha\beta\gamma\delta \ud F\alpha k \ud F\beta i \ud F\gamma l \ud F\delta j
\left(\langle A^{ij},A^{kl}\rangle-\uu gkl\langle A^{ip},\du Apj\rangle\right)\nonumber\\
&&+2\dddd R\alpha\beta\gamma\delta \uuu A\alpha kl
\left(4\udd A\beta ik\ud F\gamma l\uu F\delta i+\ud F\beta l\ud F\gamma kH^\delta
+\ud F\beta i\udd A\gamma lk\uu F\delta i\right)\nonumber\\
&&
+2\left(\nabla_\epsilon\dddd R\alpha\beta\gamma\delta+\nabla_\gamma\dddd R\alpha\delta\beta\epsilon\right)
\ud F\epsilon i\ud F\beta l\ud F\gamma k\uu F\delta i\uuu A\alpha kl\nonumber\,.
\end{eqnarray}
The second and third line can be further simplified , so that we get
\begin{eqnarray}
2\langle A,\nabla^2 \overrightarrow H\rangle
&=&\Delta|A|^2-2|\nabla^\perp A|^2\label{simons 3}\\
&&+\bigl|\langle \dd Aij,\dd Akl\rangle-\langle\dd Ail,\dd Ajk\rangle\bigr|^2+\bigl|\udd A\alpha ik\udu  A\beta jk-\udd A\beta ik\udu A\alpha jk\bigr|^2\nonumber\\
&&+2\bigl|\langle \overrightarrow H,\dd Aij\rangle-\langle\dd Aik,\du Ajk\rangle\bigr|^2-2\bigl|\langle \overrightarrow H,\dd Aij\rangle\bigr|^2\nonumber\\
&&+4\dddd R\alpha\beta\gamma\delta \ud F\alpha k \ud F\beta i \ud F\gamma l \ud F\delta j
\left(\langle A^{ij},A^{kl}\rangle-\uu gkl\langle A^{ip},\du Apj\rangle\right)\nonumber\\
&&+2\dddd R\alpha\beta\gamma\delta \uuu A\alpha kl
\left(4\udd A\beta ik\ud F\gamma l\uu F\delta i+\ud F\beta l\ud F\gamma kH^\delta
+\ud F\beta i\udd A\gamma lk\uu F\delta i\right)\nonumber\\
&&
+2\left(\nabla_\epsilon\dddd R\alpha\beta\gamma\delta+\nabla_\gamma\dddd R\alpha\delta\beta\epsilon\right)
\ud F\epsilon i\ud F\beta l\ud F\gamma k\uu F\delta i\uuu A\alpha kl\nonumber\,.
\end{eqnarray}
This last equation is useful to substitute terms in the evolution equation of $|A|^2$ (see subsection \ref{sec evolution} below).

\subsection{Special situations}
\subsubsection{Hypersurfaces}\label{sec hyper}~

If $F:M\to N$ is an immersion of a hypersurface, then $n=m+1$ and one can define a number
of scalar curvature quantities related to the second fundamental tensor of $M$. For simplicity
assume that both $M$ and $N$ are orientable (otherwise the following computations are only
local). Then there exists a unique normal vector field $\nu\in\Gamma(T^\perp M)$ - called
the principle normal - such that for all $p\in M$:
\begin{enumerate}[i)]
\item
$|\nu_{|p}|=1$, $\nu_{|p}\in T^\perp_p M$,
\item
If $e_1,\dots, e_m$ is a positively oriented basis of $T_pM$, then $$DF(e_1),\dots,DF(e_m),\nu_{|p}$$
forms a positively oriented basis of $T_{F(p)}N$.
\end{enumerate}
Using the principle normal $\nu$, one defines the (scalar) second fundamental form $h\in\Gamma(T^*M\otimes T^*M)$ by
$$h(X,Y):=\langle A(X,Y),\nu\rangle$$
and the scalar mean curvature $H$ by 
$$H:=\operatorname{trace}h$$
so that
$$A=\nu\otimes h\,,\quad\overrightarrow H=H\nu\,.$$
The map
$$\flat:TM\to T^*M\,,\quad V\mapsto V_\flat:=\langle V,\cdot\rangle$$
is a bundle isomorphism with inverse denoted by 
$$\sharp:T^*M\to TM\,.$$
This musical isomorphism can be used to define the Weingarten map
$$\mathscr{W}\in\operatorname{End}(TM)\,,\quad \mathscr{W}(X):=(h(X,\cdot))^\sharp\,.$$
Since $h$ is symmetric, the Weingarten map is self-adjoint and the real eigenvalues of $\mathscr{W}$
are called principle curvatures, often denoted by $\lambda_1,\dots,\lambda_m$, so that e.g.
$H=\lambda_1+\dots+\lambda_m$. Note, that in the theory of mean curvature flow $H$ is not
the arithmetic means $\frac{1}{m}\sum_{i=1}^m\lambda_i$ (which would justify its name)
as is often the case in classical books on differential geometry. In local coordinates we have
$$\udd A\alpha ij=\nu^\alpha\dd hij$$
and then the equations of Gau\ss\ and Codazzi can be rewritten in terms of $\dd hij$. E.g.
since $|\nu|^2=1$ we have $\langle\nabla_i\nu,\nu\rangle=0$ and then
$$\nabla_i\nu=\langle\nabla_i\nu,F^m\rangle F_m=-\langle\nu,\nabla_iF^m\rangle F_m=-\du himF_m\,.$$
This implies
\begin{eqnarray}
\nabla_i\udd A\alpha jk&=&\nabla_i(\nu^\alpha\dd hjk)\nonumber\\
&=&-\du him\dd hjk \ud F\alpha m+\nabla_i\dd hjk\nu^\alpha\,.\nonumber
\end{eqnarray}
Multiplying with $\nu_\alpha$ yields
$$\langle\nabla_i\dd Ajk,\nu\rangle=\nabla_i\dd hjk\,.$$
Interchanging $i, j$ and subtracting gives
\begin{eqnarray}
\nabla_i\dd hjk-\nabla_j\dd hik&=&\langle\nabla_i\dd Ajk-\nabla_j\dd Aik,\nu\rangle\nonumber\\
&\overset{(\ref{struc 2l})}{=}&\dddd R\alpha\beta\gamma\delta\nu^\alpha\ud F\beta k\ud F\gamma i\ud F\delta j=R^N(\nu,F_k,F_i,F_j)\,.\nonumber
\end{eqnarray}
Similarly we get Gau\ss\ equation in the form
$$\dddd Rijkl=R^N(F_i,F_j,F_k,F_l)+\dd hik\dd hjl-\dd hil\dd hjk$$
and since the codimension is one, we do not have a Ricci equation in this case.

\subsubsection{Lagrangian submanifolds}~

Let $(N,g=\langle\cdot,\cdot\rangle,J)$ be a K\"ahler manifold, i.e. $J\in\operatorname{End}(TN)$ 
is a parallel complex structure compatible with $g$.
Then $N$ becomes a symplectic manifold with the symplectic form $\omega$ given by the K\"ahler form
$\omega(V,W)=\langle JV,W\rangle$.  An immersion $F:M\to N$ is called Lagrangian, if $F^*\omega=0$ and
$n=\dim N=2m=2\dim M$.  For a Lagrangian immersion we define a section
$$\nu\in\Gamma(T^\perp M\otimes T^*M)\,,\quad\nu:=JDF\,,$$
where $J$ is applied to the $F^*TN$-part of $DF$.
$\nu$ is a $1$-form with values in $T^\perp M$ since by the Lagrangian condition $J$ induces
a bundle isomorphism (actually even a bundle isometry) between $DF(TM)$ and $T^\perp M$.
In local coordinates $\nu$ can be written as
$$\nu=\nu_idx^i=\ud\nu\alpha i\frac{\partial}{\partial y^\alpha}\otimes dx^i$$
with
$$\nu_i=JF_i=\ud J\alpha\beta \ud F\beta i\frac{\partial}{\partial y^\alpha}\,,\quad
\ud\nu\alpha i=\ud J\alpha\beta \ud F\beta i\,.$$
Since $J$ is parallel, we have
$$\nabla\nu=J\nabla DF=JA\,.$$
In contrast to hypersurfaces, we may now define a second fundamental form as a tri-linear form
$$h(X,Y,Z):=\langle\nu(X),A(Y,Z)\rangle\,.$$
It turns out that $h$ is fully symmetric. Moreover, taking a trace, we obtain a $1$-form $H\in\Omega^1(M)$, called
the mean curvature form,
$$H(X):=\operatorname{trace}h(X,\cdot,\cdot)\,.$$
In local coordinates
$$h=\ddd hijk dx^i\otimes dx^j\otimes dx^k,\,\quad H=H_idx^i,\,\quad H_i=\uu gkl\ddd hikl\,.$$
The second fundamental tensor $A$ and the mean curvature vector $\overrightarrow H$
can be written in the form
$$\udd A\alpha ij=\ddu hijk\ud\nu\alpha k\,,\quad\overrightarrow H=H^k\nu_k\,.$$
Since $J$ gives an isometry between the normal and tangent bundle of $M$, the equations
of Gau\ss\ and Ricci coincide, so that we get the single equation 
$$\dddd Rijkl=R^N(F_i,F_j,F_k,F_l)+\ddd hikm\ddu hjlm-\ddd hilm\ddu hjkm\,.$$
Since $\nabla J=0$ and $J^2=-\operatorname{Id}$ we also get
$$\nabla_i\ud\nu\alpha j
=\nabla_i(\ud J\alpha\beta\ud F\beta j)
=\ud J\alpha\beta\nabla_i\ud F\beta j
=\ud J\alpha\beta\udd A\beta ij
=\ud J\alpha\beta\ud\nu\beta k\ddu hijk
=-\ddu hijk\ud F\alpha k.$$
Similarly as above we conclude
\begin{eqnarray}
\nabla_i\ddd hjkl-\nabla_j\ddd hikl&=&\nabla_i\langle\dd Ajk,\nu_l\rangle-\nabla_j\langle\dd Aik,\nu_l\rangle\nonumber\\
&\overset{(\nabla\nu_l\in DF(TM))}{=}&\langle\nabla_i\dd Ajk-\nabla_j\dd Aik,\nu_l\rangle\nonumber\\
&\overset{(\ref{struc 2l})}{=}&R^N(\nu_l,F_k,F_i,F_j)\,.\nonumber
\end{eqnarray}
Taking a trace over $k$ and $l$, we deduce
$$\nabla_iH_j-\nabla_jH_i=R^N(\nu_k, F^k,F_i,F_j)$$
and if we take into account that $N$ is K\"ahler and $M$ Lagrangian, then the RHS is a Ricci curvature,
so that the exterior derivative $dH$ of the mean curvature form $H$ is given by
$$(dH)_{ij}=\nabla_i H_j- \nabla_jH_i=-\operatorname{Ric}^N(\nu_i,F_j).$$

If $(N,g,J)$ is K\"ahler-Einstein, then $H$ is closed (since $\operatorname{Ric}^N(\nu_i,F_j)=c\cdot\omega(F_i,F_j)=0$) and defines a cohomology
class on $M$. In this case any (in general only locally defined) function $\alpha$ with $d\alpha=H$ 
is called a Lagrangian angle. In some sense the Lagrangian condition is an integrability condition.
If we represent a Lagrangian submanifold locally as the graph over its tangent space, then the $m$ ``height" functions are not completely independent but are related
to a common potential. An easy way to see this, is to consider a locally defined $1$-form $\lambda$
on $M$ (in a neighborhood of some point of $F(M)$) with $d\lambda=\omega$. Then by the
Lagrangian condition
$$0=F^*\omega=F^*d\lambda=dF^*\lambda\,.$$
So $F^*\lambda$ is closed and by Poincar\'e's Lemma locally integrable. By the implicit function theorem this potential for $\lambda$ is related to the height functions of $M$ (cf. \cite{Smoczyk99}). Note also that by a result of Weinstein for any Lagrangian embedding $M\subset N$
there exists a tubular neighborhood of $M$ which is
symplectomorphic to $T^*M$ with its canonical symplectic structure $\omega=d\lambda$ induced by the Liouville form $\lambda$.

\subsubsection{Graphs}~

$Let  (M,g^M)$, $(K,g^K)$ be two Riemannian manifolds and $f:M\to K$ a smooth map. $f$ induces a graph 
$$F:M\to N:=M\times K\,,\quad F(p):=(p,f(p))\,.$$
Since $N$ is also a Riemannian manifold equipped with the product metric $g=g^M\times g^K$ one
may consider the geometry of such graphs. It is clear that the geometry of $F$ must be completely
determined by $f$, $g^M$ and $g^K$. Local coordinates $(x^i)_{i=1,\dots,m}$, $(z^A)_{A=1,\dots, k}$ for $M$ resp. $K$ induce local coordinates $(y^\alpha)_{\alpha=1,\dots ,n=m+k}$ on $N$
by $y=(x,z)$. Then locally
$$F_i(x)=\frac{\partial}{\partial x^i}+\ud fAi(x)\frac{\partial}{\partial z^A}\,,$$
where similarly as before $f^A=z^A\circ f\circ x^{-1}$ and $\ud fAi=\frac{\partial f^A}{\partial x^i}$. For the induced metric $F^*g=\dd gij dx^i\otimes dx^j$ we get
$$\dd gij=g^M_{ij}+g^K_{AB}\ud fAi\ud fBj\,.$$
Since this is obviously positive definite and $F$ is injective, graphs $F:M\to M\times K$ 
of smooth mappings $f:M\to K$ are always embeddings. From the formula for $DF=F_idx^i$ and the Gau\ss\ formula
one may then compute the second fundamental tensor $A=\nabla DF$. Since the precise formula
for $A$ is not important in this article, we leave the details as an
exercise to the reader.

\section{General results in higher codimension}\label{sec general}
In this section we focus on results in mean curvature flow that are valid in
any dimension and codimension and that do not depend on specific geometric situations.
\subsection{Short-time existence and uniqueness}\label{sec existence}~

Consider the mean curvature vector field $\overrightarrow H=\overrightarrow H[F]$ as an operator on the class of smooth immersions
$$\mathscr{I}:=\{F:M\to N:F\text{ is a smooth immersion}\}\,.$$
We want to compute the linearized operator belonging to $\overrightarrow H$. To this
end we need to look at the symbol and therefore we consider the locally defined expression
$$L^{\alpha;ij}_\beta:=\frac{\partial H^\alpha}{\partial F^\beta_{ij}}\,,$$
where $F^\beta_{ij}$ is shorthand for $\frac{\partial^2 F^\beta}{\partial x^i\partial x^j}$
and locally $\overrightarrow H=H^\alpha\frac{\partial}{\partial y^\alpha}$.

Let $\dd gk{i,j}:=\partial \dd gki/\partial x^j$.
We start with
\begin{eqnarray}
\frac{\partial g_{kt,m}}{\partial F^\beta_{ij}}
&=&\frac{\partial}{\partial F^\beta_{ij}}\left(g_{\delta\epsilon,\rho}\ud F\delta k \ud F\epsilon t \ud F\rho m+ g_{\delta\epsilon}(F^\delta_{km}\ud F\epsilon t+\ud F\delta kF^\epsilon_{tm})\right)\nonumber\\
&=&g_{\beta\epsilon}\ud \delta jm(\ud F\epsilon t\ud\delta ik+\ud F\epsilon k\ud \delta it)\,.\nonumber
\end{eqnarray}
From this we then obtain
\begin{eqnarray}
\frac{\partial \Gamma^s_{km}}{\partial F^\beta_{ij}}&=&\frac{1}{2} g^{st}g_{\beta\epsilon}\left((\ud\delta ik\ud\delta jm+\ud\delta im\ud\delta jk)\ud F\epsilon t\right.\nonumber\\
&&\left.+(\ud\delta it\ud\delta jm-\ud\delta im\ud\delta jt)\ud F\epsilon k
+(\ud\delta it\ud\delta jk-\ud\delta ik\ud\delta jt)\ud F\epsilon m\right)\,.\nonumber
\end{eqnarray}
Since by Gau\ss' formula
$$H^\alpha=\uu gkm \udd A\alpha km=\uu gkm(F^\alpha_{km}-\Gamma^s_{km}\ud F\alpha s
+\Gamma^\alpha_{\beta\gamma}\ud F\beta k \ud F\gamma m)$$
we obtain
\begin{eqnarray}
L^{\alpha;ij}_\beta
&=&\uu gkm\Bigl(\ud\delta\alpha\beta\ud\delta ik\ud\delta jm-\frac{1}{2} g^{st}g_{\beta\epsilon}\bigl((\ud\delta ik\ud\delta jm+\ud\delta im\ud\delta jk)\ud F\epsilon t
\nonumber\\
&&+(\ud\delta it\ud\delta jm-\ud\delta im\ud\delta jt)\ud F\epsilon k
+(\ud\delta it\ud\delta jk-\ud\delta ik\ud\delta jt)\ud F\epsilon m\bigr)
\ud F\alpha s\Bigr)\nonumber\\
&=&\ud\delta\alpha\beta\uu gij-\uu gst\dd g\beta\epsilon\uu gij \ud F\epsilon t \ud F\alpha s
-(\uu gkj\uu gsi-\uu gki\uu gsj)\dd g\beta\epsilon \ud F\epsilon k\ud F\alpha s\,.\nonumber
\end{eqnarray}

For an arbitrary nonzero $1$-form $\xi=\xi_idx^i$ we define the endomorphism $L=(\ud L\alpha\beta)_{\alpha,\beta=1,\dots,n}$ by
$$\ud L\alpha\beta:=L^{\alpha;ij}_\beta\xi_i\xi_j\,.$$
We compute
\begin{eqnarray}
\ud L\alpha\beta&=&(\ud\delta\alpha\beta-\dd g\beta\epsilon\uu gst \ud F\epsilon t \ud F\alpha s)|\xi|^2\,.\nonumber
\end{eqnarray}
Applying this to a tangent vector $F_l=\ud F\beta l\frac{\partial}{\partial y^\beta}$ we get
$$\ud L\alpha\beta \ud F\beta l=0\,.$$
If $\nu=\nu^\beta\frac{\partial}{\partial y^\beta}$ is normal, then
$$\dd g\beta\epsilon \nu^\beta \ud F\epsilon t=0$$
and hence
$$\ud L\alpha\beta \nu^\beta =|\xi|^2\nu^\alpha\,.$$
Consequently $L$ is degenerate along tangent directions of $F$ and elliptic along normal directions, more precisely for $\xi\in T_p^*M$ we have
$$L_{|p}=|\xi|^2\pi_{|p}\,,$$
where $\pi_{|p}:T_{F(p)}N\to T_p^\perp M$ is the projection of $T_{F(p)}N$ onto $T_p^\perp M$.
The reason for the $m$ degeneracies is the following: Writing a solution $F:M\to N$ of $\overrightarrow H=0$
locally as the graph over its tangent plane at $F(p)$, we see that we need as many height functions as there are codimensions, i.e. we need $k=n-m$ functions. On the other hand the system $H^\alpha=0,\alpha=1,\dots, n$ consists of $n$ coupled equations and is therefore overdetermined with a
redundancy of $m$ equations. These $m$ redundant equations
correspond
to the diffeomorphism group of the underlying $m$-dimensional manifold $M$. This
means the following:
\begin{proposition}[Invariance under the diffeomorphism group]\label{prop diff}
If $F:M\times[0,T)\to N$ is a solution of the mean curvature flow,
and $\phi\in \operatorname{Diff}(M)$ a fixed diffeomorphism of $M$,
then $\tilde F:M\times[0,T)\to N$, $\tilde F(p,t):=F(\phi(p),t)$
is another solution. In particular, the (immersed) submanifolds $\tilde M_t:=\tilde F(M,t)$
and $M_t:=F(M,t)$ coincide for all $t$.
\end{proposition}

Thus the mean curvature flow is a (degenerate) quasilinear parabolic evolution equation.
The following theorem is well-known and in particular forms a special case of a theorem by Richard Hamilton \cite{Ham82-2}, based on the Nash-Moser implicit function theorem treated in another paper by Hamilton \cite{Ham82-1}.
\begin{proposition}[Short-time existence and uniqueness]\label{prop short time}
Let $M$ be a  smooth closed manifold and $F_0:M\to N$ a smooth immersion into a smooth Riemannian manifold $(N,g)$. Then the mean curvature flow admits a unique smooth
solution on a maximal time interval $[0,T)$, $0<T\le\infty$. 
\end{proposition}

Besides the invariance of the equation under the diffeomorphism group of $M$ the flow is 
isotropic, i.e. invariant under isometries of the ambient space. This property follows from the
invariance of the first and second fundamental forms under isometries.
\begin{proposition}[Invariance under isometries]\label{theo isom}
Suppose $F:M\times[0,T)\to N$ is a smooth solution of the mean curvature flow and assume that
$\phi$ is an isometry of the ambient space $(N,g)$. Then the family $\tilde F:=\phi\circ F$
is another smooth solution of the mean curvature flow. In particular, if the initial immersion is invariant under $\phi$,
then it will stay invariant for all $t\in[0,T)$.
\end{proposition}

We note that the short-time existence and uniqueness result stated
above is not in the most general form. For example, it is not necessary
to assume smoothness initially, it suffices to assume Lipschitz continuity. We note also that in general the short-time existence
and uniqueness result for non-compact complete manifolds $M$ is open but there exist
important contributions in special cases. Based on interior estimates, Ecker and Huisken \cite{E-H91} proved  -
requiring only a local Lipschitz condition for the initial hypersurface -,
a short-time existence result for the mean curvature flow of complete hypersurfaces. In that paper the authors also show that the mean curvature flow smoothes out Lipschitz hypersurfaces (i.e. the solution becomes smooth for $t>0$). This short-time existence result has been improved
in a paper by Colding and Minicozzi \cite{Colding-Minicozzi} where one only needs to assume
a local bound for the initial height function.
The smoothing out result by Ecker and Huisken has been extended by Wang to any dimension and codimension in \cite{Wang04-2} provided the submanifolds have a small local Lipschitz  norm (which cannot be improved by an example of Lawson and Osserman) and the
ambient space has bounded geometry.
Recently 
Chen and Yin \cite{Chen-Yin07} proved that uniqueness for complete
manifolds $M$ still holds within the class of smooth solutions with bounded second fundamental tensor, if the ambient Riemannian manifold $(N,g)$ has bounded geometry in a certain sense.
Chen and Pang \cite{Chen-Pang09} considered uniqueness of unbounded solutions of the Lagrangian mean curvature flow equation for graphs. 

\subsection{Evolution equations}\label{sec evolution}~

Suppose $F:M\times[0,T)\to N$ is a smooth solution of the mean curvature flow
$$\frac{d}{dt}F=\overrightarrow H\,.$$
In this subsection we want to state and prove evolution equations of the most important
geometric quantities on $M$, like the first and second fundamental forms. 

To this end we will compute evolution equations for various sections $\sigma$ in vector bundles $E$ over
$M$. We will use the index notation introduced in subsection \ref{sec loc}.
In particular, we will consider those cases, where $\sigma$ is a section in a 
vector bundle $E_t$  which itself depends on time $t$. If for example $\nu_t$ is the 
principal normal vector field of a hypersurface $F:M\to N$, then $\nu_t$ is a section in $E_t:=F^{*}_tTN$. In
this case the mere computation of the total derivative of $\nu_t$ w.r.t. $t$ will be
insufficient since this would only make sense in local coordinates (local in space and time). To overcome this difficulty we just need to define a connection $\nabla$ on 
$F^{*}TN$, where $F$ 
is now considered as a smooth map (in general no immersion) from the space-time manifold $M\times[0,T)$ to 
$N$. A time derivative then becomes a covariant derivative in direction of $\dt$,
for example for a time dependent section $\nu\in F^{*}TN$ we have in local coordinates
$$\nu(x,t)=\nu^\alpha(x,t)\frac{\partial}{\partial y^\alpha}$$
$$\nabla_{\dt}\nu=\left(\frac{d\nu^\alpha}{dt}
+\Gamma^\alpha_{\beta\delta}\frac{dF^\beta}{dt}\nu^\delta\right)
\frac{\partial}{\partial y^\alpha}=\left(\frac{d\nu^\alpha}{dt}
+\Gamma^\alpha_{\beta\delta}H^\beta\nu^\delta\right)
\frac{\partial}{\partial y^\alpha}\,,$$
where $\Gamma^\alpha_{\beta\delta}$ are the Christoffel symbols of the Levi-Civita
connection on $N$ and $(y^\alpha)$ are local coordinates on $N$.
On the other hand, if $\sigma$ is a section in a bundle $E$ and $E$ does not depend on $t$, 
then the covariant derivative $\nabla_{\dt}\sigma$ coincides with $\dt\sigma$. For example for
the induced metric $F_t^*g\in\Gamma(T^*M\otimes T^*M)$ we have
$$F^*_tg=\dd gij(x,t) dx^i\otimes dx^j$$
and
$$\nabla_{\dt}F^*_t g=\dt\dd gij(x,t) dx^i\otimes dx^j$$
since $T^*M$ does not depend on $t$.
Likewise, for the second fundamental tensor $A$ (considered as a section in $F^*TN\otimes T^*M\otimes T^*M$, which makes sense since for $\tilde M= M\times [0,T)$ we have $T^*\tilde M= T^*M\oplus T^*\real{}$) we get
\begin{equation}\label{cov 1}
\nabla_{\dt}\udd A\alpha ij=\dt \udd A\alpha ij+\Gamma^\alpha_{\beta\gamma}\frac{dF^\beta}{dt}\udd A\gamma ij=
\dt \udd A\alpha ij+\Gamma^\alpha_{\beta\gamma}H^\beta \udd A\gamma ij\,.
\end{equation}
\begin{lemma}
If $F:M\times[0,T)\to (N,g)$ evolves under the mean curvature flow, then the induced Riemannian metrics $F_t^*g=\dd gij(x,t) dx^i\otimes dx^j\in\Gamma(T^*M\otimes T^*M)$ evolve according to
\begin{equation}\label{evol 2}
\nabla_{\dt}\dd gij=\dt\dd gij=-2\langle \overrightarrow H,A_{ij}\rangle\,.
\end{equation}
\end{lemma}
\begin{proof}
We have
$$\dd gij=\dd g\alpha\beta \ud F\alpha i \ud F\beta j$$
and thus
\begin{eqnarray}
\nabla_{\dt}\dd gij&=&\underbrace{\nabla_\gamma\dd g\alpha\beta}_{=0}\frac{d F^\gamma}{dt}\ud F\alpha i\ud F\beta j+\dd g\alpha\beta\left(\nabla_{\dt}\ud F\alpha i \ud F\beta j
+\ud F\alpha i\nabla_{\dt}\ud F\beta j\right)\nonumber\\
&=&\dd g\alpha\beta\left(\nabla_i\frac{d F^\alpha}{dt}\ud F\beta j+\ud F\alpha i\nabla_j\frac{d F^\beta}{dt}\right)\nonumber\\
&=&\dd g\alpha\beta\left(\nabla_iH^\alpha \ud F\beta j+\ud F\alpha i\nabla_jH^\beta\right)\,,\label{evol 2b}
\end{eqnarray}
where we have used that $\nabla_\gamma\dd g\alpha\beta=0$ (since $\nabla $ is metric)
and $\nabla_{\dt}\ud F\alpha i=\nabla_i\frac{dF^\alpha}{dt}$. This last identity holds since
the second fundamental tensor $\tilde A\in\Gamma(F^*TN\otimes T^*\tilde M\otimes T^*\tilde M)$
of the map $F:\tilde M\to N$ is symmetric, so that
$$\tilde A\left(\frac{\partial}{\partial x^i},\dt\right)=\nabla_i\frac{d F^\alpha}{dt}\frac{\partial}{\partial y^\alpha}=\nabla_{\dt}\ud F\alpha i\frac{\partial}{\partial y^\alpha}=\tilde A\left(\dt,\frac{\partial}{\partial x^i}\right)\,.$$
Now since $\dd g\alpha\beta H^\alpha \ud F\beta j=0$, we get
\begin{eqnarray}
0&=&\nabla_i(\dd g\alpha\beta H^\alpha \ud F\beta j)\nonumber\\
&=&\nabla_\gamma\dd g\alpha\beta \ud F\gamma iH^\alpha \ud F\beta j+\dd g\alpha\beta(\nabla_iH^\alpha \ud F\beta j+H^\alpha\nabla_i\ud F\beta j)\nonumber\\
&=&\dd g\alpha\beta(\nabla_iH^\alpha \ud F\beta j+H^\alpha \udd A\beta ij)\nonumber
\end{eqnarray}
since $\nabla_i\ud F\beta j=\udd A\beta ij$.
If we insert this into equation (\ref{evol 2b}), then we obtain the result.
\end{proof}
\begin{corollary}
The induced volume form $d\mu_t$ on $M$ evolves according to
\begin{equation}\label{evol 3}
\nabla_\dt d\mu_t=\dt d\mu_t=-|\overrightarrow H|^2d\mu_t\,.
\end{equation}
\end{corollary}
\begin{proof}
In local coordinates we have
$$d\mu_t=\sqrt{\det\dd gkl}dx^1\wedge\dots\wedge dx^m\,.$$
Since 
$$\dt\left(\det\dd gkl\right)=\left(\uu gij\dt\dd gij\right)\det\dd gkl$$
the claim follows easily.
\end{proof}
\begin{corollary}
The Christoffel symbols $\Gamma^k_{ij}$ of the Levi-Civita connection on $M$ evolve according to
\begin{eqnarray}
\dt\Gamma^k_{ij}=-\uu gkl\Bigl(\nabla_i\langle \overrightarrow H,A_{jl}\rangle+\nabla_j\langle \overrightarrow H,A_{il}\rangle-\nabla_l\langle\overrightarrow H,A_{ij}\rangle\Bigr)\,.
\end{eqnarray}
\end{corollary}
\begin{proof}
This follows directly from
$$\Gamma^k_{ij}=\frac{1}{2}\uu gkl\bigl(g_{il,j}+g_{jl,i}-g_{ij,l}\bigr)\,,$$
the evolution equation of the metric
and the fact that $\dt\Gamma^k_{ij}$ is a tensor (though $\Gamma^k_{ij}$ is not).
\end{proof}

Next we compute the evolution equation for the second fundamental tensor
$A=\udd A\alpha ij\frac{\partial}{\partial y^\alpha}\otimes dx^i\otimes dx ^j$
\begin{lemma}
The second fundamental tensor $A$ evolves under the mean curvature flow by
\begin{equation}\label{evol sec}
\nabla_{\dt}\udd A\alpha ij=\nabla_i\nabla_jH^\alpha-C^k_{ij}\ud F\alpha k+\uddd R\alpha\delta\gamma\epsilon \ud F\delta jH^\gamma \ud F\epsilon i\,,
\end{equation}
where $C^k_{ij}=\dt\Gamma^k_{ij}$.
\end{lemma}
\begin{proof}
Since
$$\udd A\alpha ij=\frac{\partial^2 F^\alpha}{\partial x^i\partial x^j}
-\Gamma^k_{ij}\ud F\alpha k
+\Gamma^\alpha_{\beta\gamma}\ud F\beta i\ud F\gamma j$$
we get
\begin{eqnarray}
\dt A^\alpha_{ij}=&&\frac{\partial^2 H^\alpha}{\partial x^i\partial x^j}
-\Gamma^k_{ij}\frac{\partial H^\alpha}{\partial x^k}
+\Gamma^\alpha_{\beta\gamma}\left(\frac{\partial H^\beta}{\partial x^i}\ud F\gamma j+
\ud F\beta i\frac{\partial H^\gamma}{\partial x^j}\right)\nonumber\\
&&-\dt\Gamma^k_{ij}\ud F\alpha k+\Gamma^\alpha_{\beta\gamma,\delta}H^\delta
\ud F\beta i\ud F\gamma j\,.
\end{eqnarray}

To continue we need some covariant expressions. For a section 
$V=V^\alpha\frac{\partial}{\partial y^\alpha}\in \Gamma\left(F^{*}TN\right)$
we have
$$\nabla_jV^\alpha=
\frac{\partial V^\alpha}{\partial x^j}+\Gamma^\alpha_{\beta\gamma}
\ud F\beta jV^\gamma$$
and then
\begin{eqnarray}
\nabla_i\nabla_jV^\alpha
&=&\frac{\partial}{\partial x^i}
\left(\frac{\partial V^\alpha}{\partial x^j}+\Gamma^\alpha_{\beta\gamma}
\ud F\beta jV^\gamma\right)
-\Gamma_{ij}^k\left(\frac{\partial V^\alpha}{\partial x^k}+\Gamma^\alpha_{\beta\gamma}
\ud F\beta kV^\gamma\right)\nonumber\\
&&+\Gamma^\alpha_{\beta\gamma}\ud F\beta i
\left(\frac{\partial V^\gamma}{\partial x^j}+\Gamma^\gamma_{\delta\epsilon}
\ud F\delta jV^\epsilon\right)\nonumber\\
&=&\frac{\partial ^2V^\alpha}{\partial x^i\partial x^j}
+\Gamma^\alpha_{\beta\gamma,\delta}\ud F\delta i\ud F\beta jV^\gamma
+\Gamma^\alpha_{\beta\gamma}\frac{\partial^2 F^\beta}{\partial x^i\partial x^j}V^\gamma
+\Gamma^\alpha_{\beta\gamma}\ud F\beta j\frac{\partial V^\gamma}{\partial x^i}\nonumber\\
&&-\Gamma_{ij}^k\left(\frac{\partial V^\alpha}{\partial x^k}+\Gamma^\alpha_{\beta\gamma}
\ud F\beta kV^\gamma\right)+\Gamma^\alpha_{\beta\gamma}\ud F\beta i
\left(\frac{\partial V^\gamma}{\partial x^j}+\Gamma^\gamma_{\delta\epsilon}
\ud F\delta jV^\epsilon\right)\nonumber\\
&=&\frac{\partial ^2V^\alpha}{\partial x^i\partial x^j}
-\Gamma_{ij}^k\frac{\partial V^\alpha}{\partial x^k}
+\Gamma^\alpha_{\beta\gamma}\left(\frac{\partial V^\beta}{\partial x^i}\ud F\gamma j+
\ud F\beta i\frac{\partial V^\gamma}{\partial x^j}\right)\nonumber\\
&&+\Gamma^\alpha_{\beta\gamma}\frac{\partial^2 F^\beta}{\partial x^i\partial x^j}V^\gamma
-\Gamma^k_{ij}\Gamma^\alpha_{\beta\gamma}
\ud F\beta kV^\gamma
+\Gamma^\alpha_{\beta\epsilon}\Gamma^\beta_{\delta\gamma}
\ud F\epsilon i
\ud F\delta jV^\gamma
\nonumber\\
&&+\Gamma^\alpha_{\beta\gamma,\delta}\ud F\delta i\ud F\beta jV^\gamma
\nonumber\\
&=&\frac{\partial ^2V^\alpha}{\partial x^i\partial x^j}
-\Gamma_{ij}^k\frac{\partial V^\alpha}{\partial x^k}
+\Gamma^\alpha_{\beta\gamma}\left(\frac{\partial V^\beta}{\partial x^i}\ud F\gamma j+
\ud F\beta i\frac{\partial V^\gamma}{\partial x^j}\right)\nonumber\\
&&+\Gamma^\alpha_{\beta\gamma}V^\gamma \udd A\beta ij
+\left(\Gamma^\alpha_{\beta\epsilon}\Gamma^\beta_{\delta\gamma}
-\Gamma^\alpha_{\beta\gamma}\Gamma^\beta_{\delta\epsilon}\right)
\ud F\epsilon i
\ud F\delta jV^\gamma
+\Gamma^\alpha_{\beta\gamma,\delta}\ud F\delta i\ud F\beta jV^\gamma\,,\nonumber
\end{eqnarray}
where we have used $\Gamma^\alpha_{\beta\gamma}=\Gamma^\alpha_{\gamma\beta}$ several times.

Applying this to $V^\alpha=H^\alpha$ we conclude
\begin{eqnarray}
\dt \udd A\alpha ij
&=&\nabla_i\nabla_jH^\alpha-\dt\Gamma^k_{ij}\ud F\alpha k+\Gamma^\alpha_{\beta\gamma,\delta}H^\delta
\ud F\beta i\ud F\gamma j\nonumber\\
&&-\Gamma^\alpha_{\beta\gamma}H^\gamma \udd A\beta ij
-\left(\Gamma^\alpha_{\beta\epsilon}\Gamma^\beta_{\delta\gamma}
-\Gamma^\alpha_{\beta\gamma}\Gamma^\beta_{\delta\epsilon}\right)
\ud F\epsilon i\ud F\delta jH^\gamma
-\Gamma^\alpha_{\beta\gamma,\delta}\ud F\delta i\ud F\beta jH^\gamma\nonumber\nonumber\\
&=&\nabla_i\nabla_jH^\alpha-\dt\Gamma^k_{ij}\ud F\alpha k
-\Gamma^\alpha_{\beta\gamma}H^\gamma \udd A\beta ij\nonumber\\
&&
+\left(\Gamma^\alpha_{\epsilon\delta,\gamma}
-\Gamma^\alpha_{\gamma\delta,\epsilon}
-\Gamma^\alpha_{\beta\epsilon}\Gamma^\beta_{\delta\gamma}
+\Gamma^\alpha_{\beta\gamma}\Gamma^\beta_{\delta\epsilon}\right)
\ud F\epsilon i\ud F\delta jH^\gamma\nonumber\\
&=&\nabla_i\nabla_jH^\alpha
-\dt\Gamma^k_{ij}\ud F\alpha k
-\Gamma^\alpha_{\beta\gamma}H^\gamma \udd A\beta ij
+\uddd R\alpha\delta\gamma\epsilon
\ud F\epsilon i\ud F\delta jH^\gamma\,.\nonumber
\end{eqnarray}
The result then follows from (\ref{cov 1}).
\end{proof}

\begin{corollary}
Under the mean curvature flow the mean curvature satisfies the following evolution equations:
\begin{eqnarray}
\nabla_\dt H^\alpha&=&\Delta H^\alpha-\uu gij C^k_{ij}\ud F\alpha k+\uddd R\alpha\delta\gamma\epsilon \ud F\epsilon i \uu F\delta i H^\gamma
+2\langle A_{kl}, \overrightarrow H\rangle \uuu A\alpha kl\quad\quad\label{evol mean1}\\
\nabla_\dt|\overrightarrow H|^2&=&\Delta|\overrightarrow H|^2-2|\nabla \overrightarrow H|^2+4\langle A^{ij},\overrightarrow H\rangle\langle A_{ij},\overrightarrow H\rangle\nonumber\\
&&+2\dddd R\alpha\beta\gamma\delta H^\alpha \ud F\beta i H^\gamma \uu F\delta i\\\label{evol mean2}
&=&\Delta|\overrightarrow H|^2-2|\nabla^\perp \overrightarrow H|^2+2\langle A^{ij},\overrightarrow H\rangle\langle A_{ij},\overrightarrow H\rangle\nonumber\\
&&+2\dddd R\alpha\beta\gamma\delta H^\alpha \ud F\beta i H^\gamma \uu F\delta i\label{evol mean3}
\end{eqnarray}
\end{corollary}
\begin{proof}
The first equation follows from
$H^\alpha=\uu gij\udd A\alpha ij$, equations  (\ref{evol 2}), (\ref{evol sec})  and
$$\nabla_\dt\uu gij=-\uu gik\uu gjl\nabla_\dt\dd gkl\,.$$
The second equation then follows from $|\overrightarrow H|^2=\dd g\alpha\beta H^\alpha H^\beta$,
$\dd g\alpha\beta \ud F\alpha i H^\beta=0$ and
$$\nabla_\dt\dd g\alpha\beta=\nabla_\gamma\dd g\alpha\beta H^\gamma=0\,.$$
Finally, (\ref{evol mean3}) follows from
\begin{eqnarray}
\nabla_k \overrightarrow H
&=&\nabla^\perp_k \overrightarrow H+\uu gij\langle\nabla_k \overrightarrow H, F_i\rangle F_j\nonumber\\
&=&\nabla^\perp_k \overrightarrow H-\uu gij\langle \overrightarrow H,\nabla_k F_i\rangle F_j\nonumber\\
&=&\nabla^\perp_k \overrightarrow H-\uu gij\langle\overrightarrow H, A_{ki}\rangle F_j\nonumber
\end{eqnarray}
and $\langle\nabla_k^\perp\overrightarrow H,F_j\rangle =0$.
\end{proof}
From the evolution equation of $\udd A\alpha ij$ we obtain in the same way
\begin{eqnarray}
\nabla_\dt|A|^2
&=&2\langle \nabla^2\overrightarrow H,A\rangle+4\langle\overrightarrow H,A^{ij}\rangle\langle A_{ik}, \du A jk\rangle\nonumber\\
&&+2\dddd R\alpha\beta\gamma\delta \uuu A\alpha kl\ud F\beta k H^\gamma \ud F\delta l\label{evol sec2}
\end{eqnarray}
Applying Simons' identity (\ref{simons 3}) we get
\begin{eqnarray}
\nabla_\dt|A|^2
&=&\Delta|A|^2-2|\nabla^\perp A|^2\nonumber\\
&&+\bigl|\langle \dd Aij,\dd Akl\rangle-\langle\dd Ail,\dd Ajk\rangle\bigr|^2+\bigl|\udd A\alpha ik\udu  A\beta jk-\udd A\beta ik\udu A\alpha jk\bigr|^2\nonumber\\
&&+2\bigl|\langle \overrightarrow H,\dd Aij\rangle-\langle\dd Aik,\du Ajk\rangle\bigr|^2-2\bigl|\langle \overrightarrow H,\dd Aij\rangle\bigr|^2\nonumber\\
&&+4\dddd R\alpha\beta\gamma\delta \ud F\alpha k \ud F\beta i \ud F\gamma l \ud F\delta j
\left(\langle A^{ij},A^{kl}\rangle-\uu gkl\langle A^{ip},\du Apj\rangle\right)\nonumber\\
&&+2\dddd R\alpha\beta\gamma\delta \uuu A\alpha kl
\left(4\udd A\beta ik\ud F\gamma l\uu F\delta i+\ud F\beta l\ud F\gamma kH^\delta
+\ud F\beta i\udd A\gamma lk\uu F\delta i\right)\nonumber\\
&&
+2\left(\nabla_\epsilon\dddd R\alpha\beta\gamma\delta+\nabla_\gamma\dddd R\alpha\delta\beta\epsilon\right)
\ud F\epsilon i\ud F\beta l\ud F\gamma k\uu F\delta i\uuu A\alpha kl\nonumber\\
&&+4\langle\overrightarrow H,A^{ij}\rangle\langle A_{ik}, \du A jk\rangle\nonumber\\
&&+2\dddd R\alpha\beta\gamma\delta \uuu A\alpha kl\ud F\beta k H^\gamma \ud F\delta l\nonumber\\
&=&\Delta|A|^2-2|\nabla^\perp A|^2\nonumber\\
&&+2\bigl|\langle \dd Aij,\dd Akl\rangle\bigr|^2+\bigl|\udd A\alpha ik\udu  A\beta jk-\udd A\beta ik\udu A\alpha jk\bigr|^2\nonumber\\
&&+4\dddd R\alpha\beta\gamma\delta \ud F\alpha k \ud F\beta i \ud F\gamma l \ud F\delta j
\left(\langle A^{ij},A^{kl}\rangle-\uu gkl\langle A^{ip},\du Apj\rangle\right)\nonumber\\
&&+2\dddd R\alpha\beta\gamma\delta \uuu A\alpha kl
\left(4\udd A\beta ik\ud F\gamma l\uu F\delta i
+\ud F\beta i\udd A\gamma lk\uu F\delta i\right)\nonumber\\
&&
+2\left(\nabla_\epsilon\dddd R\alpha\beta\gamma\delta+\nabla_\gamma\dddd R\alpha\delta\beta\epsilon\right)
\ud F\epsilon i\ud F\beta l\ud F\gamma k\uu F\delta i\uuu A\alpha kl\nonumber
\end{eqnarray}
Thus we have shown
\begin{corollary}
Under the mean curvature flow the quantity $|A|^2$  satisfies the following evolution equation:
\begin{eqnarray}
\nabla_\dt|A|^2
&=&\Delta|A|^2-2|\nabla^\perp A|^2\label{evol sec3}\\
&&+2\bigl|\langle \dd Aij,\dd Akl\rangle\bigr|^2+\bigl|\udd A\alpha ik\udu  A\beta jk-\udd A\beta ik\udu A\alpha jk\bigr|^2\nonumber\\
&&+4\dddd R\alpha\beta\gamma\delta \ud F\alpha k \ud F\beta i \ud F\gamma l \ud F\delta j
\left(\langle A^{ij},A^{kl}\rangle-\uu gkl\langle A^{ip},\du Apj\rangle\right)\nonumber\\
&&+2\dddd R\alpha\beta\gamma\delta \uuu A\alpha kl
\left(4\udd A\beta ik\ud F\gamma l\uu F\delta i
+\ud F\beta i\udd A\gamma lk\uu F\delta i\right)\nonumber\\
&&
+2\left(\nabla_\epsilon\dddd R\alpha\beta\gamma\delta+\nabla_\gamma\dddd R\alpha\delta\beta\epsilon\right)
\ud F\epsilon i\ud F\beta l\ud F\gamma k\uu F\delta i\uuu A\alpha kl\,.\nonumber
\end{eqnarray}
\end{corollary}
These general evolution equations simplify in more special geometric situations. E.g., if
the codimension is one, then $\udd A\alpha ij=\nu^\alpha\dd hij$ (cf. subsection \ref{sec hyper})
implies $|\nabla^\perp A|^2=|\nabla h|^2$, $|A|^2=|h|^2$ and
\begin{eqnarray}
\nabla_\dt|h|^2&=&\Delta|h|^2-2|\nabla h|^2+2|h|^2(|h|^2+\overline{\operatorname{Ric}}(\nu,\nu))\nonumber\\
&&-4(\uu hij\du hjm\dddu{\bar R}mlil-\uu hij\uu hlm\dddd{\bar R}milj)\nonumber\\
&&+2\uu hij(\bar\nabla_j\dddu{\bar R}0lil+\bar\nabla_l\dddu{\bar R}0ijl)\,,\label{eq sechyper}
\end{eqnarray}
where
$$\dddd{\bar R}milj:=\dddd R\alpha\beta\gamma\delta \ud F\alpha m\ud F\beta i\ud F\gamma l\ud F\delta j,\quad\overline{\operatorname{Ric}}(\nu,\nu):=\dddd R\alpha\beta\gamma\delta \nu^\alpha\ud F\beta i\nu^\gamma\uu F\delta i$$
and
$$\bar\nabla_l\dddu{\bar R}0ijl:=\nabla_\alpha\dddd R\beta\gamma\delta\epsilon\ud F\alpha l
\nu^\beta\ud F\gamma i\ud F\delta j\uu F\epsilon l.$$
Equation (\ref{eq sechyper}) is Corollary 3.5 (ii) in \cite{Huisken86}. Note that there is a plus sign in the last
line of (\ref{eq sechyper}) since our unit normal is inward pointing and the unit normal in
\cite{Huisken86} is outward directed.

\subsection{Long-time existence}\label{sec longtime}~

In general long-time existence of solutions cannot be expected as the following 
well-known theorem shows:
\begin{proposition}
Suppose $F_0:M\to \real{n}$ is a smooth immersion of a closed $m$-dimensional manifold $M$.
Then the maximal time $T$ of existence of a smooth solution $F:M\times[0,T)\to\real{n}$
of the mean curvature flow with initial immersion $F_0$ is finite.
\end{proposition}
\begin{proof}
The proof easily follows by applying the parabolic maximum principle to the function
$f:=|F|^2+2mt$ which satisfies the  evolution equation
$$\frac{d}{dt}f=\Delta f\,.$$
Hence $T\le\frac{1}{2m}\max|F_0|^2$ and the inequality is sharp since equality is attained for round spheres
centered at the origin.
\end{proof}
This result is no longer true for complete submanifolds since for example for entire
$m$-dimensional graphs in $\real{m+1}$ one has long-time existence (see
\cite{Ecker-Huisken89}). In addition, the result can fail, if the ambient space is
a Riemannian manifold since in some cases one gets long-time existence and convergence (for example in \cite{Grayson89}, \cites{Smoczyk02, Smoczyk04}, \cite{Smoczyk-Wang02}, 
\cite{Wang02}, \cite{Tsui-Wang04}).

The next well known theorem holds in any case:
\begin{proposition}\label{theo blowup}
Let $M$ be a closed manifold and $F:M\times[0,T)\to (N,g)$ a smooth solution of the mean curvature flow in a complete (compact or non-compact) Riemannian manifold $(N,g)$. Suppose the maximal time of existence $T$ is finite. Then
$$\limsup_{t\to T}\max_{M_t}|A|^2=\infty\,.$$
Here, $M_t:=F(M,t)$.
\end{proposition}
\begin{remark}
The same result also holds in some other situations. For example one can easily see that under suitable assumptions on the solution one can allow $N$ to have boundary.
\end{remark}
\begin{proof}
The theorem is one of the ``folklore" results in mean curvature flow for which
a rigorous proof in all dimensions and codimensions has not been written up
in detail but can be carried out in the same way as the corresponding proof for
hypersurfaces. This has been done by Huisken in \cites{Huisken84, Huisken90} and is again based on the maximum principle. The key observation is, that all higher derivatives $\nabla^kA$ of the
second fundamental tensor are uniformly bounded, once $A$ is uniformly bounded. This can be shown by induction and has originally been carried out for hypersurfaces using $L^p$-estimates in \cite{Huisken84}. For compact
hypersurfaces there exists a more direct argument involving the maximum principle applied to the evolution equations of $|A|^2$ in (\ref{evol sec3})
and $|\nabla^kA|^2$. The method can be found in the proof of Proposition 2.3 in \cite{Huisken90} and works in the same way in any codimension and in any
ambient Riemannian manifold with bounded geometry. 
\end{proof}

A corollary is
\begin{corollary}\label{cor long time}
Let $M$ be a closed manifold and $F:M\times[0,T)\to N$ a smooth solution of
the mean curvature flow on a maximal time interval in a  complete Riemannian
manifold $(N,g)$. If $\sup_{t\in[0,T)}\max_{M_t}|A|<\infty$, then $T=\infty$.
\end{corollary}
Note that long-time existence does not automatically imply convergence. For example, consider
the surface of revolution $N\subset\real{3}$ generated by the function $f(x)=1+e^{-x}$.
A circle $\gamma$ of revolution moving by curve shortening
flow on $N$ will then exist for all $t\in[0,\infty)$ with
uniformly bounded curvature but it will not converge since it tends off to infinity. Some results on the
regularity of curve shortening  flow in high codimension have been derived in \cite{Chen-Ma07}.

However, in some geometries once long-time existence is established one can use the Arzela-Ascoli theorem to extract  convergent subsequences.

\subsection{Singularities}\label{sec singularities}~

If a solution $F:M\times[0,T)\to N$ of the mean curvature flow exists only for finite time, then Proposition \ref{theo blowup} implies the formation of a singularity. The question then
arises how to understand the geometric and analytic nature of these
singularities. From Proposition \ref{theo blowup} we know that
$$\limsup_{t\to T}\max_{M_t}|A|^2=\infty\,.$$
One possible approach to classify singularities is to distinguish them by
the blow-up rate of $\max_{M_t}|A|^2$. The next definition originally appeared
in \cite{Huisken90} in the context of hypersurfaces in $\real{m+1}$ but can be
stated in the same way for arbitrary mean curvature flows.
\begin{definition}\label{def blowup}
Suppose $F:M\times[0,T)\to N$ is a smooth solution of the mean curvature flow with $T<\infty$
and
$$\limsup_{t\to T}\max_{M_t}|A|^2=\infty\,.$$ 
\begin{enumerate}[a)]
\item
A point $q\in N$ is called a blow-up point, if there exists a point $p\in M$ such that
$$\lim_{t\to T}F(p,t)=q\,,\quad\lim_{t\to T}|A(p,t)|=\infty\,.$$
\item
One says that $M$ develops a singularity of Type I, if there exists a constant $c>0$
such that
$$\max_{M_t}|A|^2\le\frac{c}{T-t}\,,\forall t\in[0,T)\,.$$
Otherwise one calls the singularity of Type II.
\end{enumerate}
\end{definition}
So if $q$ is a blow-up point then for $t\to T$ a singularity of Type I or Type II will form at $q\in N$ 
(and perhaps at other points as well).

In this context it is worth noting that the flow need not have a blow-up point in the sense of Definition
\ref{def blowup}, even if the second fundamental form blows up, e.g. the ambient space
might have boundary or the singularity might form at spatial infinity. For this and other reasons it 
is  appropriate to come up with more definitions. In \cite{Stone}, Stone introduced special and general
singular points.
\begin{definition}
\begin{enumerate}[a)]
\item
A point $p\in M$ is called a special singular point of the mean curvature flow, as $t\to T$, if there exists a sequence of times $t_k\to T$, such that
$$\limsup_{k\to\infty}|A|(p,t_k)=\infty.$$
\item
A point $p\in M$ is called a general singular point of the mean curvature flow, as $t\to T$, if there exists a sequence of times $t_k\to T$ and a sequence of points $p_k\to p$, such that
$$\limsup_{k\to\infty}|A|(p_k,t_k)=\infty.$$
\end{enumerate}
\end{definition}
The reason to introduce the blow-up rate in Definition \ref{def blowup}
is that for closed
submanifolds in euclidean space one always has an analogue inequality in the other direction, i.e.
\begin{equation}\label{eq lowest}
\max_{M_t}|A|^2\ge\frac{\tilde c}{T-t}
\end{equation}
for some positive number $\tilde c$ (note that this does not necessarily hold, if the ambient
space $N$ differs from $\real{n}$). So in some sense singularities of Type I have the best controlled
blow-up rate of $|A|^2$. Because of (\ref{eq lowest})
one may actually refine the definition of special and general
singular points for the mean curvature flow in $\real{n}$, as was originally done by Stone in \cite{Stone}.
Instead of requiring $\limsup_{k\to\infty}|A|(p_k,t_k)=\infty$ one can define a general
singular point $p\in M$ such that there exists some $\delta>0$ and a sequence $(p_k,t_k)\to (p,T)$
with
$$|A|^2(p_k,t_k)\ge\frac{\delta}{T-t_k}.$$
A sequence $(p_k,t_k)$ with this property is called an essential blow-up sequence.
Although (\ref{eq lowest}) gives a minimum blow-up rate for $\max_{p\in M}|A|^2(p,t)$ in the euclidean space,
as $t$ approaches $T$, this does not rule out the possibility that, while $|A|^2(p,t)\ge\frac{\delta}{T-t}$ in some part of $M$, the blow-up of $|A|^2$ might simultaneously occur at some slower rate (say like $(T-t)^{-\alpha}, \alpha\in(0,1)$) somewhere else. Such "slowly forming singularities" would not be detected
by a Type I blow-up procedure (see below) since the rescaling would be too fast. It is therefore
interesting to understand, if this phenomenon occurs at all. As was recently shown by Le and Sesum 
\cite{Le-Sesum10b} this does not happen in the case of Type I singularities of hypersurfaces in
$\real{m+1}$ and all notions of singular sets defined in \cite{Stone} coincide. In particular they
prove that the blow-up rate of the mean curvature must coincide with the blow-up rate of the
second fundamental form, if a singularity of Type I is forming. We also mention that there exist
many similarities between the formation of singularities in mean curvature flow and Ricci flow
(see \cite{Enders-Mueller-Topping} for a nice overview on Type I singularities in Ricci flow).

{\bf Type I:}
Let us now assume that $q\in\real{n}$ is a blow-up point of Type I of 
$F:M\times[0,T)\to\real{n}$ and that $\dim M=m$. Huisken introduced the following rescaling
technique in \cite{Huisken90} for hypersurfaces, 
but obviously it can be done in the same way for any codimension in $\real{n}$:
Define an immersion $\tilde F:M\times[-1/2\log T,\infty)\to\real{n}$ by
$$\tilde F(\cdot,s):=(2(T-t))^{-1/2}(F(\cdot,t)-q)\,,\quad s(t)=-\frac{1}{2}\log(T-t)\,.$$
One can then compute that $\tilde F$ satisfies the rescaled flow equation
$$\frac{d}{ds}\tilde F=\tilde{\overrightarrow H}+\tilde F\,.$$
Since by assumption $|A|^2\le c/(T-t)$ the second fundamental tensor $\tilde A$ of the
rescaling is uniformly bounded in space and time. To study the geometric and analytic behavior
of the rescaled immersions $\tilde M_s=\tilde F(M,s)$, Huisken proved a monotonicity
formula for hypersurfaces in $\real{n}$ moving by mean curvature. The corresponding result
in arbitrary dimension and codimension is as follows:  For $t_0\in\real{}$ let
$$\rho:\real{n}\times\real{}\setminus\{t_0\}:=\frac{1}{(4\pi(t_0-t))^{\frac{m}{2}}}e^{-\frac{|y|^2}{4(t_0-t)}}\,.$$
Then $\rho_{|\real{m}\times\real{}\setminus\{t_0\}}$ is the backward heat kernel of $\real{m}$ at
$(0,t_0)$ and the following monotonicity formula holds
\begin{proposition}[Monotonicity formula (cf. Huisken \cite{Huisken90})]
Let $F:M\times[0,T)\to\real{n}$ be a smooth solution of the mean curvature flow and let $M$ be
closed and $m$-dimensional. Then
$$\dt\int_M\rho(F(p,t),t)d\mu(p,t)=-\int_M\left|\overrightarrow H(p,t)+\frac{F^\perp(p,t)}{2(t_0-t)}\right|^2\rho(F(p,t),t)d\mu(p,t)\,,$$
where $d\mu(\cdot,t)$ denotes the volume element on $M$ induced by the immersion
$F(\cdot,t)$ and $F^\perp$ denotes the normal part of the position vector $F$.
\end{proposition}
The proof is a simple consequence of
$$\dt\rho=\left(\frac{m}{2(t_0-t)}-\frac{|F|^2}{4(t_0-t)^2}-\frac{\langle F,\overrightarrow H\rangle}{2(t_0-t)}\right)\rho$$
and
$$\Delta\rho=\left(-\frac{m}{2(t_0-t)}+\frac{|F^\top|^2}{4(t_0-t)^2}-\frac{\langle F,\overrightarrow H\rangle}{2(t_0-t)}\right)\rho$$
so that by the divergence theorem and from $\dt d\mu=-|\overrightarrow H|^2d\mu$ we get
$$\dt\int_M\rho d\mu=\int_M(\dt\rho+\Delta\rho-|\overrightarrow H|^2\rho)d\mu=-\int_M\left|\overrightarrow H+\frac{F^\perp}{2(t_0-t)}\right|^2\rho d\mu\,.$$
Though the proof is easy, it is not obvious to look at the backward heat kernel when studying
the mean curvature flow. This nice formula was used by Huisken to study the asymptotic behavior of
the Type I blow-up and he proved the following beautiful theorem for hypersurfaces which again
holds in arbitrary codimension
\begin{proposition}[Type I blow-up (cf. Huisken \cite{Huisken90})]\label{theo 1blowup}
Suppose $F:M\times[0,T)\to\real{n}$ is a smooth solution of the mean curvature flow
of a closed $m$-dimensional smooth manifold $M$. Further assume that $T<\infty$ is finite and that $0\in\real{n}$ is a Type I blow-up point as $t\to T$. Then for any sequence $s_j$  there is a subsequence $s_{j_k}$ such that the
rescaled immersed submanifolds $\tilde M_{s_{j_k}}$ converge smoothly to an immersed nonempty
limiting submanifold $\tilde M_\infty$. Any such limit satisfies the equation
\begin{equation}\label{eq self}
\tilde {\overrightarrow H}+\tilde F^\perp=0\,.
\end{equation}
\end{proposition}
Note that by Proposition \ref{theo isom} it is no restriction to assume that the blow-up point
coincides with the origin.
In general the limiting submanifold $\tilde M_\infty$ need not have the same topology
as $M$, for example compactness might no longer hold. In addition it is unclear, if all solutions of
(\ref{eq self}) occur as blow-up limits of Type I singularities of compact submanifolds.

A solution of (\ref{eq self}) is called a self-similar shrinking solution (or self-shrinker for short)
of the mean curvature flow.
Namely, one easily proves that a solution of (\ref{eq self}) 
shrinks homothetically under the mean curvature flow and that there is a smooth 
positive function $c$ explicitly computable from the initial data and depending on 
the rescaled time $s$  such that 
$$\tilde {\overrightarrow H}_s+c(s)\tilde F^\perp_s=0\,.$$

There exists another interesting class of self-similar solutions of the mean curvature flow. These
are characterized by the elliptic equation
\begin{equation}\label{eq selfexp}
\overrightarrow H-F^\perp=0
\end{equation}
and are called self-expanders.
In \cite{Ecker-Huisken89} Ecker and Huisken proved that entire graphs in $\real{m+1}$
(in codimension $1$) approach asymptotically expanding self-similar solutions if they 
satisfy a certain growth condition at infinity. Later Stavrou \cite{Stavrou98} proved
the same result under the weaker assumption that the graph has bounded gradient and a unique cone at infinity. Furthermore, he gave a characterization of expanding self-similar solutions to mean curvature flow with bounded gradient.

A classification of self-shrinking or self-expanding solutions is far from being complete. 
However there are
some special situations for which one can say something. Self-shrinking curves have
been completely classified by Abresch and Langer in \cite{Abresch-Langer86}. Though their
proof has been carried out for the curve shortening flow in $\real{2}$ the result also
applies to arbitrary codimension since (\ref{eq self}) becomes an ODE for $m=1$ and the
solutions are uniquely determined by their position and velocity vectors so that all $1$-dimensional
solutions of (\ref{eq self}) must be planar. For hypersurfaces there exists a beautiful theorem by 
Huisken in \cite{Huisken93} that describes all self-shrinking hypersurfaces with nonnegative
(scalar) mean curvature. Later this result could be generalized by the author in the following sense 

\begin{proposition}[\cite{Smoczyk05}]
For a closed immersion $M^m\subset\real{n}$, $m\ge 2$
are equivalent:
\begin{enumerate}[a)]
\item $M$ is a self-shrinker of the mean curvature flow with nowhere vanishing mean curvature
vector $\overrightarrow H$ and the principal normal vector $\nu:=\overrightarrow H/|\overrightarrow H|$ is parallel in the normal bundle.
\item $M$ is a minimal immersion in a round sphere.
\end{enumerate}
\end{proposition}

In the same paper one finds a similar description for the non-compact case.

 Type I singularities usually occur when there
exists some kind of pinching of the second fundamental form and such situations occur quite  often
(cf. subsection \ref{sec special}). It is therefore surprising that there are situations, where one can exclude
Type I singularities at all. In \cite[Theorem 2.3.5]{Smoczyk99}  it was shown that there do not exist any
compact Lagrangian solutions of (\ref{eq self}) with trivial Maslov class $m_1=[H/\pi]=0$.
Wang \cite{Wang01} and Chen \& Li \cite{Chen-Li02} observed that finite time Type I singularities of the Lagrangian
mean curvature flow of
closed Lagrangian submanifolds can be excluded,
if the initial Lagrangian is almost calibrated in the sense that $*\operatorname{Re}(dz_{|M})>0$.
The condition to be almost calibrated is equivalent to the assumption that the Maslov class 
is trivial and that the Lagrangian angle $\alpha$ satisfies $\cos\alpha>0$. The difference of
the results of Wang, Chen and Li in \cites{Wang01, Chen-Li02} w.r.t.
the result in \cite{Smoczyk99} is, that the blow-up need not be compact
any more. 
Later Neves \cite{Neves07} extended this result to the case of zero Maslov class, i.e. to the case
where a globally defined Lagrangian angle $\alpha$ exists on $M$, thus removing the almost calibrated condition. In \cite[Theorem 1.9]{GSSZ07} we proved a classification result for 
Lagrangian self-shrinkers and expanders in case they are entire graphs with a growth condition
at infinity. In these cases Lagrangian self-similar solutions must be minimal Lagrangian cones.

Therefore when we study
the Lagrangian mean curvature flow of closed Lagrangian submanifolds with trivial Maslov class
we need to consider singularities of Type II only.

{\bf Type II:}
To study the shape of the submanifold near a singularity of Type II one can define a different
family of rescaled flows.  Following an idea of Hamilton \cite{Hamilton93} one can choose
a sequence $(p_k,t_k)$ as follows: For any integer $k\ge 1$ let $t_k\in[0,T-1/k], p_k\in M$
be such that
$$|A(p_k,t_k)|^2(T-\frac{1}{k}-t_k)=\max_{\tiny\begin{matrix}t\le T-1/k\\ p\in M\end{matrix}}|A(p,t)|^2(T-\frac{1}{k}-t)\,.$$
Furthermore one chooses
$$L_k=|A(p_k,t_k)|\,,\quad\alpha_k=-L_k^2t_k\,,\quad\omega_k=L_k^2(T-t_k-1/k)\,.$$
If the singularity is of Type II then one has
$$t_k\to T\,,\quad L_k\to\infty\,,\quad\alpha_k\to-\infty\,,\quad\omega_k\to\infty\,.$$
Instead of $|A|$ one may use other quantities in the definition of these sequences, if it's known that they blow-up
with a certain rate as $t\to T$.
For example, in \cite{Huisken-Sinestrari99} the mean curvature $H$ was used in
the case of mean convex hypersurfaces in $\real{m+1}$.

Then one can consider the following rescaling:
For any $k\ge 1$, let $M_{k,\tau}$ be the family of submanifolds defined by the
immersions
$$F_k(\cdot,\tau):=L_k(F(\cdot,L_k^{-2}\tau+t_k)-F(p_k,t_k))\,,\quad\tau\in[\alpha_k,\omega_k]\,.$$
The proper choice of the blow-up quantity ($|A|, H$ or similar) in the definition of the rescaling 
will be essential to describe its behavior. Besides this rescaling technique there exist other methods
to rescale singularities and the proper choice of the rescaling procedure depends on the particular
situation in which the flow is considered. A nice reference for some of the scaling techniques is \cite{Ecker04}.

If $M$ is compact and develops a Type II singularity
then a subsequence of the flows $M_{k,\tau}$ converges smoothly to an eternal mean curvature
flow $\tilde M_\tau$ defined for all $\tau\in\real{}$. Then a classification of Type II singularities
depends on the classification of eternal solutions of the mean curvature flow.

In $\real{2}$ the only convex eternal solution (up to scaling)
of the mean curvature flow is given by the ``grim reaper"
$$y=-\log\cos x/\pi\,.$$
The grim reaper is a translating soliton of the mean curvature flow, i.e. it satisfies the geometric
PDE
$$\overrightarrow H=V^\perp\,,$$
for some fixed vector $V\in\real{n}$.
A translating soliton moves with constant speed in direction of  $V$.

In \cite{Angenent-Velazquez97} the authors constructed some particular solutions of the mean curvature flow that develop Type II singularities. In $\real{2}$ examples of curves that
develop a Type II singularity are given by some cardioids \cite{Angenent91}.
Using a Harnack inequality, Hamilton \cite{Hamilton95b} proved that any eternal convex solution of the mean curvature
flow of hypersurfaces in $\real{m+1}$ must be a translating soliton, if it
assumes its maximal curvature at some point in space-time.
In \cite{Chen-Jian-Liu05} the authors study whether such convex translating solutions
are rotationally symmetric, and if every
2-dimensional rotationally symmetric translating soliton is strictly convex.

Various different notions of weak solution have been developed to extend the flow beyond the
singular time $T$, including the geometric measure theoretic solutions of Brakke 
\cite{Brakke} and the level set solutions of Chen, Giga \& Goto \cite{Chen-Giga-Goto91}
and Evans \& Spruck \cite{Evans-Spruck91}, which were subsequently studied
further by Ilmanen \cite{Ilmanen92}. In \cite{Huisken-Sinestrari09}
Huisken and Sinestrari define such a notion based on a surgery procedure.

\section{Special results in higher codimension}\label{sec special}
In this chapter we mention the most important results in mean curvature
flow that depend on more specific geometric situations and we will focus
on results in higher codimension, especially on graphs and results in Lagrangian mean curvature flow.
\subsection{Preserved classes of immersions}
\begin{definition}
Let $\mathscr{I}$ be the class of smooth $m$-dimensional immersions into a Riemannian manifold $(N,g)$
and suppose $\mathscr{F}\subset\mathscr{I}$ is a subclass. We say that $\mathscr{F}$
is a preserved class under the mean curvature flow, if for any solution
$F_t:M\to N$, $t\in[0,T)$ of the mean curvature flow with $(F_0:M\to N)\in\mathscr{F}$ we also have $(F_t:M\to N)\in\mathscr{F}$
for all $t\in[0,T)$.
\end{definition}

Preserved classes of the mean curvature flow are very important since one can often prove special results within these classes. Many classes can be expressed in terms of algebraic properties
of the second fundamental form and in general it is a hard problem to detect those classes.
We give a number of examples
\begin{example}~

\begin{enumerate}[a)]
\item
$\mathscr{F}_1:=\{\text{Convex hypersurfaces in }\real{m+1}\}$
\item
$\mathscr{F}_2:=\{\text{Mean convex hypersurfaces in }\real{m+1}\text{, i.e. $H>0$}\}$
\item
$\mathscr{F}_3:=\{\text{Embedded hypersurfaces in Riemannian manifolds}\}$
\item
$\mathscr{F}_4:=\{\text{Hypersurfaces in $\real{m+1}$ as entire graphs over a flat plane}\}$
\item
$\mathscr{F}_5:=\{\text{Lagrangian immersions in K\"ahler-Einstein manifolds}\}$
\end{enumerate}
\end{example}
To prove that classes are preserved one often uses the parabolic maximum principle (at least in the compact case). Besides the classical maximum principle for scalar quantities there exists an important maximum principle for bilinear forms
due to Richard Hamilton that was originally proven in \cite{Ham82-2}
and improved in \cite{Ham86}.

Another very important property is the pinching property of certain classes of immersions in $\real{n}$.
\begin{definition}
Let $F:M\to\real{n}$ be a (smooth) immersion. We say that the second fundamental form
$A$ of $F$ is $\delta$-pinched, if the inequality
$$|A|^2\le\delta|\overrightarrow H|^2$$
holds everywhere on $M$.
\end{definition} 
From
$$0\le \left|A-\frac{1}{m}\overrightarrow H\otimes F^*g\right|^2=|A|^2-\frac{1}{m}|\overrightarrow H|^2$$
with $m=\dim M$ we immediately obtain that $\delta$ is bounded from below by $1/m$.

For hypersurfaces in $\real{m+1}$ it is known:
\begin{proposition}\label{theo 4.4}
Let $\delta\ge1/m$. The class of closed $\delta$-pinched hypersurfaces in $\real{m+1}$ is a preserved class under the mean curvature flow.
\end{proposition}
\begin{proof}
This easily follows from the maximum principle and the evolution equation for $f:=|A|^2/H^2$.
\end{proof}
It can be shown that an $m$-dimensional submanifold in $\real{n}$ is $1/m$-pinched, if and only if it is either a part of a round sphere or a flat subspace.
Therefore closed pinched submanifolds are in some sense close to spheres.
In some cases this pinching can improve under the mean curvature flow.
To explain this in more detail, we make the following definition:
Let $\mathscr{F}$ be a nonempty class of smooth $m$-dimensional immersions $F:M\to\real{n}$,
where $M$ is not necessarily fixed, and set
$$\delta_{\mathscr{F}}:=\sup\{\delta\in\real{}:|A_F(p)|^2\ge\delta|\overrightarrow H_F(p)|^2\,,\forall p\in M,\,\forall (F:M\to\real{n})\in\mathscr{F}\}\,,$$
where $A_F$ and $\overrightarrow H_F$ denote the second fundamental form and mean curvature vector of the immersion $F:M\to\real{n}$.
Then $\delta_{\mathscr{F}}\ge\frac{1}{m}$ and $\delta_{\mathscr{F}}$ is finite,
if and only if $\mathscr{F}$ contains an immersion $F:M\to\real{n}$ for which $\overrightarrow H_F$ does not vanish
completely.
\begin{definition}
Let $\mathscr{F}$ be a preserved class of smooth $m$-dimensional immersions with $\delta_{\mathscr{F}}<\infty$ and suppose $\delta$ is some real number with $\delta>\delta_{\mathscr{F}}$. We say that $\mathscr{F}$ is
$\delta$-pinchable, if for any $\epsilon$ with $0\le\epsilon<\delta-\delta_{\mathscr{F}}$ the class
$$\mathscr{F}_\epsilon:=\{(F:M\to\real{n})\in\mathscr{F}:|A_F(p)|^2\le(\delta_{\mathscr{F}}+\epsilon)|\overrightarrow H_F(p)|^2\,,\forall p\in M\}$$
is a preserved class under the mean curvature flow.
\end{definition}
\begin{example}
\begin{enumerate}[a)]
\item
It follows from Theorem \ref{theo 4.4} that the class $\mathscr{F}(m,m+1)$ of smooth $m$-dimensional closed
immersions  into $\real{m+1}$ is $\delta$-pinchable for any $\delta\ge 1/m=\delta_{\mathscr{F}(m,m+1)}$ and that the pinching constant $\delta_{\mathscr{F}(m,m+1)}$ is attained if and only
if the immersion $F:M\to\real{n}$ is a round sphere or a flat plane (or part of).
\item
A beautiful result recently obtained by Andrews and Baker \cite{Andrews-Baker09} shows that the class $\mathscr{F}(m,m+k)$
of smooth $m$-dimensional closed immersions  into $\real{m+k}$
is $\delta$-{pin\-ch\-able} with $\delta=1/({m-1})$, if $m\ge 4$ and with $\delta=4/3m$ for $2\le m\le 4$. Here $\delta_{\mathscr{F}(m,m+k)}=1/m$.
They prove that $\delta$-pinched immersions contract to round points. Thus for such immersions one has $M=S^m$ and
they are smoothly homotopic to hyperspheres.
\end{enumerate}
\end{example}
We will now show that the class $\mathscr{L}(m)$ of  smooth closed Lagrangian immersions  into $\complex{m}$ is not $\delta$-pinchable for any $\delta$.
\begin{theorem}
Let $\mathscr{L}(m)$ be the class of smooth closed Lagrangian immersions  into $\complex{m}$,
$m>1$. Then $\delta_{\mathscr{L}(m)}=3/(m+2)$ and $\mathscr{L}(m)$
is not $\delta$-pinchable for any $\delta$.
\end{theorem}
\begin{proof}
Given a Lagrangian immersion $F:M\to\complex{m}$ we have
$$0\le\left|\ddd hijk-\frac{1}{m+2}(H_i\dd gjk+H_j\dd gki+H_k\dd gij)\right|^2=|A|^2-\frac{3}{m+2}|\overrightarrow H|^2\,,$$
where $H_idx^i$ is the mean curvature form.
This implies $\delta_{\mathscr{L}(m)}\ge\frac{3}{m+2}$. On the other hand equality is attained
for flat Lagrangian planes and for the Whitney spheres. These are given by restricting the
immersions
$$\tilde F_r:\real{m+1}\to\complex{m}\,,\quad \tilde F_r(x^1,\dots,x^{m+1}):=\frac{r(1+ix^{m+1})}{1+(x^{m+1})^2}(x^1,\dots,x^{m}),\,\quad r>0$$
to $S^m\subset\real{m+1}$, i.e. $F_r:=\tilde F_{r|S^m}:S^m\to\complex{m}$ is a Lagrangian immersion
of the sphere with $|A|^2=\frac{3}{m+2}|\overrightarrow H|^2$. The number $r$ is called the radius of the Whitney sphere. This shows $\delta_{\mathscr{L}(M)}=\frac{3}{m+2}$. It has been shown by Ros and Urbano in \cite{Ros-Urbano98} that Whitney spheres and flat Lagrangian planes
are the only Lagrangian submanifolds in $\complex{m}$, $m>1$, for which $|A|^2=\frac{3}{m+2}|\overrightarrow H|^2$. Now if $\mathscr{L}(M)$ would be $\delta$-pinchable for some $\delta$, then in particular the Lagrangian
mean curvature flow would preserve the identity $|A|^2=\frac{3}{m+2}|\overrightarrow H|^2$. This is certainly
true for the flat planes but for the Whitney sphere this cannot be true. Because the
result of Ros and Urbano implies that under
the assumption of $\delta$-pinchability a Whitney sphere would then stay a Whitney sphere under the Lagrangian mean curvature flow and the radius of the spheres would decrease.
In other words, the Whitney sphere would have to be a self-similar shrinking solution
of the Lagrangian mean curvature flow. This is a contradiction to the well-known result (first shown in \cite[Corollary 2.3.6]{Smoczyk99}), that there are no self-shrinking Lagrangian spheres in $\complex{m}$, if $m>1$.
\end{proof}

\subsection{Lagrangian mean curvature flow}~

In this subsection we will assume that $F:M\to N$ is a closed smooth Lagrangian immersion into a
K\"ahler manifold $(N,g,J)$. It has been shown in \cite{Smoczyk96} that the Lagrangian
condition is preserved, if the ambient K\"ahler manifold is Einstein. This includes the important
case of Calabi-Yau manifolds, i.e. of  Ricci flat K\"ahler manifolds. Recently a generalized
Lagrangian mean curvature flow in almost K\"ahler manifolds with Einstein connections
has been defined by Wang and the author in \cite{Smoczyk-Wang09}. This generalizes an earlier result by Behrndt \cite{Behrndt08}.
The Einstein condition is
relevant in view of the Codazzi equation which implies that the mean curvature form
is closed, a necessary condition to guarantee that the deformation is Lagrangian. To explain
this in more detail, observe that the symplectic form $\omega$ induces an isomorphism between
the space of smooth normal vector fields along $M$, and the
space of smooth $1$-forms on $M$. Namely, given $\theta\in\Omega^1(M)$ there exists
a unique normal vector field $V\in\Gamma(T^\perp M)$ with $\theta=\omega(\cdot, V)$.
If $F:M\times[0,T)\to N$ is a smooth family of Lagrangian immersions evolving in
normal direction driven by some smooth time depending $1$-forms $\theta\in\Omega^1(M)$ we have
$$0=\dt F^*\omega=d(\omega(\dt F,\cdot))=-d\theta$$
and consequently $\theta$ must be closed. Since the mean curvature form is given by $$H=\omega(\cdot,\overrightarrow H)$$
we obtain that the closeness of $H$ is necessary to guarantee that the mean curvature flow preserves the Lagrangian condition, and it is indeed sufficient (\cites{Smoczyk96, Smoczyk99}).
In the non-compact case this is open in general, but in some cases (like graphs over complete
Lagrangian submanifolds with bounded geometry) this can be reduced to the existence
problem of solutions to a parabolic equation of Monge-Amp\`ere type. The Lagrangian condition can be interpreted as an integrability condition. For example, if $M$ is a graph in $\complex{m}
=\real{m}\oplus i\real{m}$ over the real part, i.e. if $M$ is the image of some embedding
$$F:\real{m}\to\complex{m}\,,\quad F(x)=x+iy(x)\,,$$
where $y=y_idx^i$ is a smooth $1$-form on $\real{m}$, then $M$ is Lagrangian
if and only if $y$ is closed. Consequently there exists a smooth function $u$ (called a generating function) such that $y=du$. Assuming that $M$ evolves under the mean curvature flow and 
that all subsequent graphs $M_t$ are still Lagrangian one can integrate the evolution
equation of $y=du$ and obtains a parabolic evolution equation of Monge-Amp\`ere type
for $u$. Conversely, given a solution $u$ of this parabolic Monge-Amp\`ere type equation 
on $\real{m}$ one can generate
Lagrangian graphs $F=(x,du)$ and it can be shown that these graphs move under the mean curvature
flow (cf. \cite{Smoczyk99}). The same principle works in a much more general context,
namely if the initial Lagrangian submanifold lies in some K\"ahler-Einstein manifold and the
Lagrangian has bounded geometry. The boundedness of the geometry is essential for the proof
since this allows to exploit the implicit function theorem to obtain the existence of a Monge-Amp\`ere
type equation similar as above.

This integrability property has one important consequence. In general, given a second order
parabolic equation, one would need uniform $C^{2,\alpha}$-bounds of the solution in space
and uniform $C^{1,\alpha}$-estimates in time to ensure long-time existence, as follows
from Schauder theory.  For the mean curvature flow these estimates are already
induced by a uniform estimate of the second fundamental form 
$A$ (see Corollary \ref{cor long time}), so essentially by $C^2$-estimates.
In the Lagrangian mean curvature  flow $F:M\times[0,T)\to N$  one may instead use the parabolic equation of Monge-Amp\`ere type for the generating function $u$ and consequently one just needs $C^{1,\alpha}$-estimates
in space and $C^{0,\alpha}$ estimates in time for $F$ which itself is of first order in $u$.
In some situations this principle has been used successfully, for example in 
\cites{Smoczyk-Wang02, Smoczyk04}. There it was shown that Lagrangian 
tori $M=T^m$  in flat tori $N=T^{2m}$ converge to flat Lagrangian
tori, if the universal cover possesses a convex generating function $u$. We also mention a recent generalization to the
complete case by Chau, Chen and He \cite{Chau-Chen-He09b}.

The evolution equations for the Lagrangian mean curvature flow have been derived in
\cite{Smoczyk99} (see also \cite{Smoczyk96}) and can also be obtained directly from our general evolution equations stated in subsection \ref{sec evolution}. Besides the evolution equation for the induced metric the equation for the mean curvature form $H=H_idx^i$
is perhaps the most important and is given by
\begin{equation}\label{eq evolmcf}
\nabla_{\dt}H=dd^\dagger H+\frac{S}{2m}H\,,
\end{equation}
where $S$ denotes the scalar curvature of the ambient K\"ahler-Einstein manifold, $m$
is the dimension of the Lagrangian immersion and $d^\dagger H=\nabla^iH_i$.
In particular it follows that the cohomology class $[He^{-\frac{S}{2m}t}]$ is invariant under
the Lagrangian mean curvature flow and in a Calabi-Yau manifold the Lagrangian immersions
with trivial first Maslov class $m_1$ (we have $m_1=\frac{1}{\pi}[H]$) form a preserved class.
This also shows that if the scalar curvature $S$ is nonnegative, then a necessary condition to
have long-time existence and smooth convergence of the Lagrangian mean curvature flow to
a minimal Lagrangian immersion is that the initial mean curvature form is exact. Exactness
of the mean curvature form will then be preserved and a globally defined Lagrangian angle $\alpha$ with $d\alpha=H$
exists for all $t$. This last result also holds for general scalar curvature $S$ and after
choosing a proper gauge for $\alpha$ one can prove \cite[Lemma 2.4]{Smoczyk99a}
that $\alpha$ satisfies the evolution
equation
\begin{equation}\label{evol angle}
\dt\alpha=\Delta\alpha+\frac{S}{2m}\alpha\,.
\end{equation}
It is then a simple consequence of the maximum principle that on compact Lagrangian submanifolds
$M$ with trivial Maslov class in a Calabi-Yau manifold  there exist uniform upper and lower bounds for the
Lagrangian angle given by its initial maximum resp. minimum. In particular, the condition
to be almost calibrated, i.e. $*\operatorname{Re}(dz_{|M})=\cos\alpha>0$ is preserved. Here $dz$ denotes
the complex volume form on the Calabi-Yau manifold and it is well known that the Lagrangian angle
$\alpha$ satisfies 
$$dz_{|M}=e^{i\alpha} d\mu\,,$$
where $d\mu$ is the volume form on $M$.
Almost calibrated Lagrangian submanifolds in Calabi-Yau manifolds have some nice properties
under the mean curvature flow. As was mentioned earlier, from the results in 
\cites{Chen-Li02, Neves07, Smoczyk99, Wang01} we know that singularities of the Lagrangian mean curvature
flow of compact Lagrangian immersions with trivial Maslov class in Calabi-Yau manifolds cannot
be of Type I and therefore a big class of singularities is excluded. So far one cannot say much about
singularities of Type II and in particular, one does not know if they occur at all in the case of compact almost calibrated
Lagrangians (though some authors have some rather heuristic arguments for the existence of such
singularities). It is worth noting that there do not exist any compact almost calibrated Lagrangian
immersions in $\real{2m}$ (but in $\mathbb{T}^{2m}$ they exist).
In \cite[Theorem 1.3]{Smoczyk02} it was shown that there exists a 
uniform (in time) lower bound for the volume
of a compact almost calibrated Lagrangian evolving by its mean curvature  in a Calabi-Yau (and more generally in a K\"ahler-Einstein
manifold of non-positive scalar curvature). 

An interesting class of Lagrangian immersions is given by monotone Lagrangians. A Lagrangian
immersion $F:M\to\real{2m}$ is called monotone, if
\begin{equation}\label{eq monotone}
[H]=\epsilon [F^*\lambda]\,,
\end{equation}
for some positive constant $\epsilon$ (called monotonicity constant). Here $\lambda$ is the
Liouville form on $\real{2m}=T\real{m}$.
In \cite{GSSZ07} we proved several theorems concerning monotone Lagrangian immersions.
From the evolution equations of $H$ and $F^*\lambda$ one derives that monotonicity is preserved
with a time dependent monotonicity constant $\epsilon(t)$.
Gromov \cite{Gromov85} proved that given an embedded Lagrangian submanifold $M$ in $\real{2m}$
there exists a holomorphic disc with boundary on $M$. On the other hand, from the evolution equations of $H$ and $F^*\lambda$ we get that the area of holomorphic discs with boundary
representing some fixed homology class in $M$ is shrinking linearly in time. If the Lagrangian is monotone,
then the shrinking rate for the area of holomorphic discs is the same for all homology classes.

Unfortunately it is unknown, if embeddedness of Lagrangian submanifolds is preserved under
mean curvature flow (in general, embeddedness in higher codimension is not preserved but self-intersection
numbers might be).  Suppose $F:M\times[0,T)\to\real{2m}$
is a Lagrangian mean curvature flow of a compact monotone Lagrangian with initial monotonicity
constant $\epsilon>0$ and suppose $0<T_{e}\le T$
is the embedding time, i.e. the maximal time such that 
$F_t:M\to\real{2m}$ is an embedding for all $0\le t<T_e$. Then we proved \cite[Theorem 1.6 and Theorem 1.11]{GSSZ07} that 
$T_e\le\frac{1}{\epsilon}\,.$
Moreover
$$T=\frac{1}{\epsilon}\,,$$
in case $T_e=T$ and if $M$ develops a Type I singularity as $t\to T$. We note that this result is rather
unique in mean curvature flow. Usually it is not possible to {explicitly} determine the span of life of
a solution and to determine it in terms of its initial data.
In the same paper we also proved the existence of compact embedded monotone Lagrangian
submanifolds (even with some additional symmetry) that develop Type II singularities and
consequently it is not true that monotone embedded Lagrangian submanifolds must
develop Type I singularities, as was conjectured earlier by some people.

Lagrangian submanifolds appear naturally in another context. If 
$$f:M\to K$$ 
is a symplectomorphism between two symplectic manifolds $(M,\omega^M) $, $(K,\omega^K)$
then the graph 
$$F:M\to M\times K\,,\quad F(p)=(p,f(p))$$
is a Lagrangian embedding in $(M\times K,(\omega^M, -\omega^K))$.

If $(M,\omega^M,J^M, g^M)$ and $(K,\omega^K,J^K,g^K)$ are both K\"ahler-Einstein, then
the product manifold is K\"ahler-Einstein as well and one can use the Lagrangian mean curvature
flow to deform a symplectomorphism. In \cite{Smoczyk02} symplectomorphisms between Riemann
surfaces of the same constant curvature $S$ have been studied and it was shown 
(Lemma 10 and Lemma 14) that Lagrangian
graphs that come from symplectomorphisms stay graphs for all time. The same result was obtained
independently by Wang in \cite{Wang01b} (the quantities $r$ in \cite[Lemma 10]{Smoczyk02}
and $\eta$ in \cite[Proposition 2.1]{Wang01b} are the same up to some positive constant). In \cite{Smoczyk02} the graphical condition was then
used in the case of non-positive curvature $S$ and under the angle condition $\cos\alpha>0$ (almost
calibrated) to derive explicit bounds for the second fundamental form and to establish
long-time existence and smooth convergence to a minimal Lagrangian surface. Wang used
the graphical condition in \cite{Wang01b} to obtain long-time existence without a
sign condition on $S$ by methods related to White's regularity theorem and then
proved convergence of subsequences to minimal Lagrangian surfaces. Later he refined his result
and proved smooth convergence in \cite{Wang08}.
In a recent paper by Medos and Wang \cite{Medos-Wang09} it is shown that symplectomorphisms
of $\mathbb{CP}^m$ for which the singular values satisfy some pinching condition
can be smoothly deformed into a biholomorphic isometry of $\mathbb{CP}^m$.

In a joint paper \cite{Smoczyk-Wang02} (see also \cite{Smoczyk04}) Wang and the author
studied Lagrangian graphs in the cotangent bundle of a flat torus and proved that Lagrangian
tori with a convex generating function converge smoothly to a flat Lagrangian torus. 
In this case the convexity of the generating function $u$
implies that the Monge-Amp\`ere type operator that appears in the evolution
equation of $u$ becomes concave and then results of Krylov \cite{Krylov87} imply
uniform 
$C^{2,\alpha}$-estimates in space and $C^{1,\alpha}$-estimates in time and long-time
existence and convergence follows. A similar result holds for  non-compact graphs \cite{Chau-Chen-He09b}.

\subsection{Mean curvature flow of graphs}~

As the results mentioned at the end of the last subsection show, mean curvature flow of graphs
behaves much ``nicer" than in the general case. There are many results for graphs moving
under mean curvature flow. The first result in this direction was the paper
by Ecker and Huisken \cite{Ecker-Huisken89} where long-time existence of entire graphs 
in $\real{m+1}$ (hypersurfaces) was shown. Convergence to flat subspaces follows, if the
growth rate at infinity is linear. Under a different growth rate they prove that the hypersurfaces converge asymptotically to entire self-expanding solutions of the mean curvature flow. The crucial observation in their paper
was that the angle function $v:=\langle\nu,Z\rangle$ (scalar product of the unit normal and
the height vector $Z$) satisfies a very useful evolution equation that can be exploited to bound
the second fundamental form appropriately.

Many results in mean curvature flow of graphs have been obtained by Wang. For
example in \cite{Wang02} he studied the graph induced by a map $f:M\to K$ between to Riemannian manifolds of constant sectional curvatures. Under suitable assumptions on the differential of $f$ and the curvatures of $M$ resp. $K$ 
he obtained long-time existence and convergence to constant maps.
In \cite{Tsui-Wang04} the authors consider a graph in the product
$M\times K$ of two Riemannian manifolds of constant sectional curvatures. 
A map $f:M\to K$  for which the singular values $\lambda_i$ of $f$ satisfy the condition
$\lambda_i\lambda_j<1$ for all $i\neq j$ is called an area decreasing map. 
The main theorem in their paper states long-time existence of the mean curvature flow and convergence to a constant map
under the following assumptions:
\begin{enumerate}[i)]
\item the initial graph of $f$ is area-decreasing;
\item $ \sigma^M\ge|\sigma^K|,\,\sigma^M+\sigma^K>0$ and $\dim M\ge 2$, 
\end{enumerate}
where $\sigma^M,\sigma^K$ denote the sectional curvatures of $M$ resp. $K$.
In particular area decreasing maps from $S^m$ to $S^k$ are homotopically trivial for $m\ge 2$.

In \cite{Li-Li03} graphs in Riemannian products of two space forms have been studied and under certain assumptions on the initial graph 
long-time existence was established.
In \cite{Wang05} two long-time existence and convergence results for the mean curvature flow of graphs induced by maps $f:M\to K$ between two compact Riemannian manifolds of dimension $m=\dim M\geq2$ and $\dim K=2$ are given.
In the first theorem $M$ and $K$ are assumed to be flat, and in the second
theorem, $M=S^m$ is an $m$-sphere of constant curvature $k_1 >0$ and $K$
a compact surface of constant curvature $k_2$ with $|k_2|\leq k_1$. 
The key assumption on the graph is expressed in terms of the Gau\ss\ map, i.e. the map which assigns to a point $p$ its tangent space. The latter is an element of the bundle of $m$-dimensional subspaces of $TN$, $N=M\times K$ and it is shown that there exists a sub-bundle $\germ{G}$ of $TN$ which is preserved along the mean curvature flow. The same author proved a beautiful general theorem for
the Gau\ss\ map under the mean curvature flow (see \cite{Wang03}).


\end{document}